\documentclass[10pt,a4paper]{amsart}
\usepackage{etex} 
\usepackage[OT2,T1]{fontenc}
\usepackage[utf8]{inputenc}
\usepackage[russian,english]{babel}
\PassOptionsToPackage{normalem}{ulem}
\usepackage{ulem} \usepackage[thinlines]{easybmat}
\usepackage{amssymb,amsthm,amsbsy,amstext,pb-diagram,lamsarrow,pb-lams,booktabs,microtype}
\usepackage[osf]{mathpazo}
\usepackage[italic]{mathastext}
\usepackage{eucal}

\usepackage{mathtools}
 \mathtoolsset{showonlyrefs=true,showmanualtags=true}

\numberwithin{equation}{section}
\numberwithin{figure}{section}
\theoremstyle{plain}

  \theoremstyle{remark}
  
  \theoremstyle{definition}
  \newtheorem*{example*}{\protect\examplename}
  \theoremstyle{definition}

\theoremstyle{plain}
\newtheorem{Theorem}{Theorem}[section]

\newtheorem{theorem}[Theorem]{Theorem}

\newtheorem{cor}[Theorem]{Corollary}
\newtheorem{prop}[Theorem]{Proposition}
\newtheorem{lemma}[Theorem]{Lemma}
\theoremstyle{definition}

\newtheorem{remark}[Theorem]{Remark}
\theoremstyle{remark}

\newcommand{\Coker}{\qopname \relax o{Coker}}

\newcommand{\GL}{\operatorname{GL}}

\newcommand{\id}{\operatorname{id}}

\renewcommand{\a}{\alpha}

\newcommand{\ad}{{\mathop{\textup{ad}}}}
\newcommand{\aND}{\quad\text{and}\quad}
\newcommand{\AND}{\qquad\text{and}\qquad}
\newcommand{\astwo}{\ast\mkern-1.1mu\ast}
\newcommand{\asthree}{\ast\mkern-1.1mu\ast\mkern-1.1mu\ast}
\newcommand{\Aut}{{\mathop{\textup{Aut}}}}
\renewcommand{\b}{\beta}
\newcommand{\B}{\mathcal B}
\newcommand{\BeE}{\text{B${}_\Ep$EO}}
\newcommand{\BeO}{\text{B${}_\Ep$O}}
\newcommand{\BebO}{{\text B{}_\Ep\mkern 2.0mu\overline{\mkern-2.0mu\text O\mkern-2.0mu}\mkern 2.0mu}}
\newcommand{\bep}{{\bar\epsilon}}
\newcommand{\bepK}{{\bar\epsilon_{\mathcal K_{\mathbf R}}}}
\newcommand{\bepR}{{\bar\epsilon_{\mathbf R}}}
\newcommand{\BGL}{\operatorname{BGL}}

\newcommand{\BSp}{\operatorname{BSp}}

\newcommand{\C}{\mathbf{C}}
\renewcommand{\c}{\circ}

\newcommand{\CCo}{$\C\x\C^\op$}
\newcommand{\Cl}{\colon}

\newcommand{\dd}{\partial}
\newcommand{\ds}{\oplus}
\newcommand{\D}{\Delta}

\newcommand{\E}{\operatorname{E}}
\newcommand{\eF}{{}_\Ep F}
\newcommand{\eH}{{}_\Ep H}
\newcommand{\eE}{{}_\Ep\mkern-3.6mu\operatorname{EO}}
\newcommand{\eO}{{}_\Ep\mkern-2.8mu\operatorname{O}}
\newcommand{\End}{{\mathop{\textup{End}}}}
\newcommand{\ep}{\epsilon}
\newcommand{\ebO}{{{}_\Ep\mkern 2.0mu\overline{\mkern-2.0mu\text O\mkern-2.0mu}\mkern 2.0mu}}
\newcommand{\Ep}{\varepsilon}
\newcommand{\esKQ}{{}_\Ep\mkern-0.6mu\mathbf{\mathop{KQ}}}
\newcommand{\esU}{{}_\Ep\mkern-0.7mu\mathbf U}
\newcommand{\esV}{{}_\Ep\mkern-2.8mu\mathbf V}
\newcommand{\eU}{{}_\Ep U}
\newcommand{\eV}{{}_\Ep\mkern-1muV}
\newcommand{\eW}{{}_\Ep\mkern-0.8muW}

\newcommand{\FE}{{\mathop{FE}}}
\newcommand{\g}{\gamma}
\newcommand{\G}{\Gamma}
\renewcommand{\H}{\mathbf{H}}
\newcommand{\hA}{\mkern4.0mu\hat{\mkern-4.0muA}}
\newcommand{\hF}{\mkern1.8mu\hat{\mkern-1.8muF}}
\newcommand{\HHo}{$\H\x\H^\op$}
\newcommand{\HRA}{\hookrightarrow}
\renewcommand{\i}{\iota}

\renewcommand{\k}{\kappa}
\newcommand{\K}{\mathcal K}
\newcommand{\Ker}{\qopname\relax o{Ker}}
\newcommand{\KQ}{{\mathop{KQ}}}
\newcommand{\eKQ}{{}_{\Ep}\KQ}
\newcommand{\Kc}{\mathcal K_{\mathbf C}}
\newcommand{\Kr}{\mathcal K_{\mathbf R}}

\renewcommand{\l}{\lambda}
\renewcommand{\L}{\mathcal L}

\newcommand{\LA}{\longrightarrow}
\newcommand{\LLA}{\longleftarrow}
\newcommand*{\longhookrightarrow}{\ensuremath{\lhook\joinrel\relbar\joinrel\rightarrow}}
\newcommand{\LHRA}{\longhookrightarrow}
\newcommand{\LMT}{\longmapsto}
\newcommand{\LL}{\left}
\newcommand{\mO}{${}_{{\text-}\text1}\mkern-1.0mu$O}
\newcommand{\mU}{${}_{{\text-}\text1}\mkern-2.2mu$U}
\newcommand{\mV}{${}_{{\text-}\text1}\mkern-4.2mu$V}

\newcommand{\N}{\mathbf{N}}

\newcommand{\OO}{O$\x$O}
\newcommand{\Om}{\Omega}
\newcommand{\OU}{\text{O$/$U}}

\newcommand{\into}{\rightarrowtail}
\newcommand{\onto}{\twoheadrightarrow}
\newcommand{\op}{{\text{\textup{op}}}}

\newcommand{\ph}{\varphi}
\newcommand{\pmO}{${}_\pm\mkern-1mu$O}
\newcommand{\pmU}{${}_\pm\mkern-1mu$U}
\newcommand{\pmV}{${}_\pm\mkern-4mu$V}
\newcommand{\pO}{${}_{\text1}\mkern-1.0mu$O}
\newcommand{\pU}{${}_{\text1}\mkern-2.2mu$U}
\newcommand{\pV}{${}_{\text1}\mkern-4.2mu$V}
\newcommand{\kf}{{k,\ph}}
\newcommand{\pT}{\widehat{\otimes}_{\pi}}
\newcommand{\Q}{\mathbf{Q}}
\newcommand{\R}{\mathbf{R}}
\newcommand{\RR}{\right}
\newcommand{\RRo}{$\R\x\R^\op$}
\newcommand{\s}{\sigma}
\newcommand{\sbeq}{\subseteq}
\newcommand{\sK}{\mathbf{K}}
\newcommand{\sKQ}{\mathbf{\mathop{KQ}}}
\newcommand{\sU}{\mathbf U}

\newcommand{\Sp}{\operatorname{Sp}}

\newcommand{\Spp}{Sp$\x$Sp}
\newcommand{\SpU}{\text{Sp$/$U}}
\newcommand{\T}{\otimes}
\newcommand{\tA}{\mkern4.0mu\tilde{\mkern-4.0muA}}
\newcommand{\tB}{\mkern0.8mu\tilde{\mkern-0.8muB}}
\newcommand{\tC}{\mkern4.0mu\tilde{\mkern-4.0muC}}
\renewcommand{\th}{\theta}
\newcommand{\ti}{\tilde}

\newcommand{\Til}{{\text{\large\textasciitilde}}}

\newcommand{\w}{\wedge}
\newcommand{\x}{\times}
\newcommand{\xT}{{\otimes}_{\max}}
\renewcommand{\t}{{\text{\textup{top}}}}
\newcommand{\UO}{\text{U$/$O}}
\newcommand{\USp}{\text{U$/$Sp}}
\newcommand{\UU}{U$\x$U}

\newcommand{\TZ}{{\otimes}_{\mathbf Z}}
\newcommand{\Z}{\mathbf{Z}}
\newcommand{\Ztwo}{\mathbf{Z}/\mkern-2mu2\mathbf{Z}}
\newcommand{\Zsixt}{\mathbf{Z}/\mkern-2mu16\mathbf{Z}}

\newcommand{\BD}{\begin{diagram}}
\newcommand{\ED}{\end{diagram}}
\newcommand{\BE}[1]{\begin{equation}\label{#1}}
\newcommand{\EE}{\end{equation}}
\newcommand{\BM}{\left(\begin{array}}
\newcommand{\EM}{\end{array}\right)}

  \providecommand{\remarkname}{\inputencoding{latin9}Remark}

  \providecommand{\remarkname}{\inputencoding{latin9}Remark}
  \providecommand{\theoremname}{\inputencoding{latin9}Theorem}

  \providecommand{\examplename}{\inputencoding{latin9}Example}
\providecommand{\theoremname}{\inputencoding{latin9}Theorem}

  \providecommand{\examplename}{\inputencoding{latin9}Example}
\providecommand{\theoremname}{\inputencoding{latin9}Theorem}

  \providecommand{\remarkname}{\inputencoding{latin9}Remark}
\providecommand{\theoremname}{\inputencoding{latin9}Theorem}

\makeatother

  \providecommand{\examplename}{\inputencoding{latin9}Example}
  \providecommand{\remarkname}{\inputencoding{latin9}Remark}
\providecommand{\theoremname}{\inputencoding{latin9}Theorem}

\newcommand{\FineBar}[3]{\mkern#1\overline{\mkern-#1{#3}\mkern-#2}\mkern#2}
\newcommand{\bA}{\FineBar{3.5mu}{0.5mu}A}
\newcommand{\bF}{\FineBar{3.0mu}{0.5mu}F}
\newcommand{\bK}{\FineBar{3.0mu}{0.5mu}K}
\newcommand{\bv}{\FineBar{1.5mu}{0.5mu}v}
\newcommand{\bw}{\FineBar{1.0mu}{0.5mu}w}
\newcommand{\ebKQ}{{}_{\Ep}\mkern 3.0mu\overline{\mkern-3.0mu\KQ\mkern-0.5mu}\mkern 0.5mu}
\newcommand{\ebU}{{}_{\Ep}\mkern 2.0mu\overline{\mkern-2.0muU\mkern-0.5mu}\mkern 0.5mu}
\newcommand{\ebV}{{}_{\Ep}\mkern 1.0mu\overline{\mkern-2.5muV\mkern-0.5mu}\mkern 0.5mu}
\newcommand{\ebW}{{}_{\Ep}\mkern 1.0mu\overline{\mkern-2.0muW\mkern-0.5mu}\mkern 0.5mu}

\begin{document}

\title{Algebraic and Hermitian $K$-theory
of $\K$-rings
}

\author{Max Karoubi}
\address{Université Denis Diderot -- Paris 7\\ Institut de Mathématiques de Jussieu Paris · Rive Gauche}
\email{max.karoubi@gmail.com}

\author{Mariusz Wodzicki}
\address{Department of Mathematics\\
University of California\\
Berkeley}
\email{wodzicki@math.berkeley.edu}
\thanks{Mariusz Wodzicki was partially supported by the NSF grant 1001846.}


\begin{abstract}
The main purpose of the present article is to establish
the real case of ``Karoubi's Conjecture'' \cite{KaroubiSLN}
in algebraic $K$-theory. The complex case was proved
in 1990-91 (cf.\ \cite{Suslin-Wodzicki} for
the $C^{\ast}$-algebraic form of the conjecture, and
\cite{Wodzicki.ECM} for both the Banach and $C^{\ast}$-algebraic
forms). 
Compared to the case of complex algebras,
the real case poses additional difficulties.
This is due to the fact that topological
$K$-theory of real Banach algebras has period 8
instead of 2. The method we employ to
overcome these difficulties can also
be used for complex algebras, and provides
some simplifications to the original proofs.
We also establish a natural analog of
Karoubi's Conjecture for Hermitian
$K$-theory.
\end{abstract}

\maketitle

\pagestyle{myheadings}

\vspace{3ex}
\section{Introduction}

\smallskip
Let $A$ be a complex $C^{\ast}$-algebra and $\K=\K(H)$ be the ideal
of compact operators on a separable complex Hilbert space $H$.
One of us conjectured around 1977 that the natural comparison
map between the algebraic and topological $K$-groups 
\[
\ep_n\Cl K_n(\K\bar{\T}A) \LA K_n^\t(\K\bar{\T}A)\simeq K_n^\t(A)
\]
is an isomorphism for a suitably completed tensor
product $\bar{\T}$. The conjecture was announced
in \cite{KaroubiSLN} where accidentally only $\K\pT A$
was mentioned, and it was proved there for
$n\leq0$. In \cite{KaroubiJ.op.theory} the conjecture was established
for $n\leq2$ and $\K\bar{\T}A$ having the meaning of the
$C^{\ast}$-algebra completion of the algebraic tensor product
(in view of \emph{nuclearity} of $\K$, there is only one
such completion).

The $C^{\ast}$-algebraic form of the conjecture was established
for all $n\in\Z$ in 1990 \cite{Suslin-Wodzicki.PNAS},
\cite{Suslin-Wodzicki}. A year later the conjecture was
proved also for all $n\in\Z$ and $\K\pT A$ where $A$
was only assumed to be a Banach $\C$-algebra with one-sided
bounded approximate identity \cite{Wodzicki.ECM}.

Here we prove analogous theorems for $K$-theory
of real Banach and $C^{\ast}$-algebras, and then
deduce similar results in Hermitian $K$-theory.

\smallskip
The article is organized as follows.
In Section \ref{s:CF} we set the stage
by introducing the concept of a \emph{$\K$-ring},
which is a slight generalization and
a modification to what was called
a ``stable'' algebra in \cite{KaroubiJ.op.theory}.
We also introduce a novel notion of
a \emph{stable retract}.
Then we proceed to demonstrate that
the comparison map between algebraic
and topological $K$-groups in degrees $n\leq0$
is an isomorphism for Banach algebras
that are stable retracts of $\K\xT A$ for some
$C^{\ast}$-algebra $A$, or of $\K\pT A$,
for some Banach algebra $A$ (Theorem \ref{th:n<0}).

We recall how to endow $K_{\ast}(\K)$ with
a canonical structure of a $\Z$-graded,
associative, graded commutative, and unital ring
in Section \ref{s:Kst C}.
If a $\K$-ring is $H$-unital as a $\Q$-algebra,
we equip its $\Z$-graded algebraic $K$-groups
with a structure of a $\Z$-graded unitary
$K_{\ast}(\K)$-module.
We distinguish two cases: $\Kr$-rings
and $\Kc$-rings, where $\Kr$ stands for the ring
of compact operators on a real Hilbert space,
and $\Kc$ --- on the complex Hilbert space.

These structures are used to prove several
results, two of them:
2-periodicity of algebraic $K$-groups
for $\Kc$-rings (Theorem \ref{th:PTC}) and
a comparison theorem for Banach $\Kc$-rings
(Theorem \ref{th:KCst}), are derived in Section 2.
Both results were known before
\cite{Wodzicki.ECM}, \cite{Suslin-Wodzicki}.

Relying on a rather
delicate argument employing $K$-theory
with coefficients mod 16 and certain
classical results of stable homotopy theory,
we detect in Section \ref{s:Kst R}
the existence of an element
$v_8\in K_8(\Kr)$ which maps onto
a generator of $K_8^\t(\Kr)\simeq\Z$.
We showed that $K_{-8}(\Kr)\simeq\Z$
in Section \ref{s:CF}. The element $v_8$
is the multiplicative inverse of
a generator of that group.
This allows us to establish 8-periodicity
of algebraic $K$-groups for arbitrary
$\Kr$-rings (Theorem \ref{th:PTR}),
and to prove that the comparison map
is an isomorphism for Banach
$\Kr$-rings
(``the real case of Karoubi's Conjecture'').
Both results are new.

The remaining sections are devoted
to Hermitian $K$-theory. The goal
is to deduce comparison theorems
in Hermitain $K$-theory from
the corresponding results in
algebraic $K$-theory (for complex
$C^{\ast}$-algebras this was partially done
in \cite{Battikh}). The primary
tool in this task is a pair of \emph{Comparison Induction
Theorems} which are among the consequences
of the ``Fundamental Theorem of Hermitian
$K$-theory'' \cite{KaroubiAnnals}.
In the same chapter we also provide
a rather thorough discussion of
Hermitian $K$-groups for nonunital rings,
we compare two approaches to defining
relative Hermitian $K$-groups
and, in the end, we deduce excision
properties in Hermitian $K$-theory from
the corresponding properties
in algebraic $K$-theory,
cf.\ \cite{BattikhCR}, \cite{Battikh}.
This is achieved with a different
kind of Induction Theorems.

We would like to point out that the results of Sections 2
through 8 hold also for the corresponding algebraic and
Hermitian $K$-groups with finite coefficients.
The hypothesis of $H$-unitality over $\Q$ can be dropped
since $\Q$-algebras satisfy excision in algebraic and
Hermitian K-theory with finite coefficients.

In two appendices that follow we collect
material frequently used throughout the article.
We review some facts about multiplicative
structures in algebraic $K$-theory
in Appendix \ref{Ap BAI}, indicating
difficulties of the nonunital case.
Appendix \ref{Ap BAI} provides
a well known characterization of
pure-exact extensions of Banach spaces,
and a much less known criterion of
pure-exactness for extensions of Banach
modules associated with ideals in Banach
algebras. Complete proofs are given
for the reader's convenience.

\smallskip
The present article was written so that
the presentation would be accessible
to experts in theory of Banach and operator
algebras. We hope to attract their attention
to the possibilities offered by the interplay
between real structures, involutive algebras,
and $K$-theory.

\smallskip
We would like to thank Thierry Fack who kindly
agreed to write an appendix where he supplies
a short proof of exactness of the maximal
tensor product on the category of
$C^{\ast}$-algebras---a result we could
not find a satisfactory reference to
in the existing literature.

\medskip
{\small
Aknowledgements. We thank C.~Anantharaman, T.~Fack,
A.~Guichardet and S.~Wassermann for discussions
about $C^{\ast}$-algebra tensor products,
and J.~A.~Neisendorfer for providing an up to date
information on multiplicative structures in
homotopy groups with finite coefficients.
}

\bigskip
\section{Natural transformations $\ep_n\Cl K_n\to K_n^\t$
 for $n\leq0$}\label{s:CF}

\smallskip
For $n<0$, the algebraic and the topological
$K$-groups both can be defined in terms of $K_0$ using
the notions of the cone and suspension functor
suitable for rings and, respectively, for Banach algebras.
The compatibility of the two constructions
leads to a sequence of natural transformations
of functors $\ep_n\Cl K_n\to K_n^\t$ and the
associated connecting homomorphisms, which means
that any extension in the category of real or
complex Banach algebras,
\BE{ABC}
\BD
 \node{A}\node{B}\arrow{w,t,A}{\pi}\node{C}\arrow{w,t,V}{\i}
\ED\;,
\EE
induces a morphism of half-infinite exact sequences
{\small
\BE{LES ABC}
\BD\dgARROWLENGTH=2.2em
 \node{\cdots}
  \node{K_n(B)}\arrow{w,t}{\pi_n}\arrow{s,l}{\ep_n}
   \node{K_n(C)}\arrow{w,t}{\i_n}\arrow{s,l}{\ep_n}
    \node{K_{n+1}(A)}\arrow{w,t}{\dd_{n+1}}\arrow{s,l}{\ep_{n+1}}
     \node{K_{n+1}(B)}\arrow{w,t}{\pi_{n+1}}\arrow{s,l}{\ep_{n+1}}
      \node{\cdots}\arrow{w,t}{\pi_{n+1}}\\
 \node{\cdots}
  \node{K_n^\t(B)}\arrow{w,t}{\pi_n^\t}
   \node{K_n^\t(C)}\arrow{w,t}{\i_n^\t}
    \node{K_{n+1}^\t(A)}\arrow{w,t}{\dd_{n+1}^\t}
     \node{K_{n+1}^\t(B)}\arrow{w,t}{\pi_{n+1}^\t}
      \node{\cdots}\arrow{w,t}{\pi_{n+1}^\t}
\ED\;,
\EE}%
where $n<0$. Below, we will refer to $\ep_n$ as the \emph{comparison
maps}.
They are defined for all $n\in\Z$, and when the comparison map
\BE{K comp}
\ep_n\Cl K_n(A)\LA K_n^\t(A)
\EE
is an isomorphism, we
say that the Banach algebra $A$ is \emph{$K_n$-stable}.

We say that a $k$-algebra $R$ is a \emph{stable retract}
of a $k$-algebra $A$ if,
for some $l\geq1$, the \emph{stabilization homomorphism}
\BE{i_l}
\i_l\Cl R\LHRA M_l(R),\qquad
r\longmapsto
 \left(
  \begin{BMAT}(rc,0pt,20pt)[4pt]{c.c}{c.c}
   r & 0_{1,l-1} \\
   0_{l-1,1} & 0_{l-1,l-1}
  \end{BMAT}
 \right)
\EE
factorizes through $A$,
\BE{i=lk}
\BD
\node{R}\arrow[2]{e,t,V}{\i_l}\arrow{se,b,,..}{\k}
 \node[2]{M_l(R)}\\
\node[2]{A}\arrow{ne,b,,..}{\nu}
\ED\;,
\EE
with $\k$ and $\nu$ being Banach algebra homomorphisms.
Note that \emph{to be a stable retract of} is a transitive
relation: if $A$ is a stable retract of $B$, then $R$ is also
a stable retract of $B$.
A stable rectract of a $K_n$-stable algebra
is automatically $K_n$-stable.

\medskip
Let us consider the $C^{\ast}$-algebra $\B(H)$ of bounded linear
operators on a Hilbert space $H$.
The space itself can be \emph{real} or \emph{complex}.
We will denote by $F$ the corresponding field of coefficients,
which is $\R$ in the former case, and $\C$ in the latter.

In the case when $H$ is infinite dimensional
$\B(H)$ is an example of what Farrell
and Wagoner call an \emph{infinite
sum ring} \cite[p.\,477]{FarrellWagoner}.
For an infinite sum ring $R$, its general
linear group $GL(R)$ is acyclic
\cite[Corollary 2.5]{Wagoner} which means
that $BGL(R)^+$ is contractible.
Since the property of being an infinite
sum ring is preserved under suspensions,
$K_n(R\T_k A)=0$ for all $n\in\Z$.
If $R$ is an infinite sum $k$-algebra,
which means that the endomorphism
\BE{r inf}
r\LMT r^\infty\qquad(r\in R),
\EE
the key ingredient of the structure
of an infinite sum ring, is $k$-linear,
then $R\T_k A$ is again an infinite sum
$k$-algebra and, accordingly, $K_n(R\T_k A)=0$
for all unital $k$-algebras
and arbitrary $n\in\Z$.

If $R$ is an infinite sum Banach algebra,
which means that endomorphism \eqref{r inf}
is continuous, then so is the completed
projective tensor product $R\pT B$
with any unital Banach algebra $B$,
and both algebraic and topological
$K$-groups of $K_n(R\pT B)$
vanish in all degrees.
Finally, when $R$ is an infinite sum
$C^{\ast}$-algebra, which means
that \eqref{r inf} is a $C^{\ast}$-endomorphism,
then so is the tensor product $R\xT C$
with any unital $C^{\ast}$-algebra $C$
(the maximal tensor product
can be replaced by any functorial
 $C^{\ast}$-product).

The algebra of bounded linear operators on
an infinite dimensional Hilbert space is
an infinite sum $C^{\ast}$-algebra.\label{BH inf sum}
In the course of this work we shall repeatedly
use the fact that
\BE{KB A=0}
K_{\ast}(\B(H)\T A)=K_{\ast}(\B(H)\pT B)=K_{\ast}(\B(H)\xT C)=0
\EE
and
\BE{KtB A=0}
K_{\ast}^\t(\B(H)\pT B)=K_{\ast}^\t(\B(H)\xT C)=0
\EE
for any unital $F$-algebra $A$, Banach algebra
$B$, and $C^{\ast}$-algebra $C$.

Vanishing of algebraic and topological $K$-groups
of $\B(H)$ and various tensor products
with unital algebras can also be established
using the language of \emph{flabby
categories} \cite{Karoubiflabby},
and has been a part of the $K$-theoretical
lore for more than four decades.

\bigskip
When $H$ is infinite dimensional and separable,
then $\K=\K(H)$ is a unique proper, nonzero and
closed ideal in $\B=\B(H)$; it contains every proper
nonzero ideal in $\B$ as a dense subset and
the quotient $\B/K$ is called
the \emph{Calkin algebra}.
Below, $H$ will always denote an infinite-dimensional
separable Hilbert space.

For any closed subspace $H'\sbeq H$, there is an
associated nonunital embedding of $C^{\ast}$-algebras
\BE{i}
\i^{H'H}\Cl\B(H')\LHRA\B(H),
 \qquad a\LMT i^{H'H}\c a\c p^{HH'},
\EE
where $p^{HH'}\Cl H\to H'$ denotes the projection
onto $H'$ with kernel $(H')^\perp$ and $i^{H'H}$
denotes the inclusion of $H'$ into $H$. It sends
$\K(H')$ to $\K(H)$.
When $H'$ is finite dimensional, then
\[
\K(H')=\B(H')=\End_FH'.
\]
In the one dimensional case we obtain a nonunital
algebra homomorphism
\BE{F->K}
\i\Cl F=\End_FH'\LHRA\K(H),\qquad 1\LMT i^{H'H}\c p^{HH'}.
\EE

Let us consider the following 2-parameter family
of $C^{\ast}$-algebras
\BE{Flm}
F_{lm}:=\K^{\xT l}\xT(\B/\K)^{\xT m}\qquad(l,m\in\N),
\EE
and, given a $C^{\ast}$-algebra $A$, the associated 2-parameter family
\BE{Alm}
A_{lm}:=F_{lm}\xT A,
\EE
where the \emph{maximal} tensor product
is formed in the category of $C^{\ast}$-algebras
over $F$.
By tensoring the extension of $C^{\ast}$-algebras
\BE{CBK}
\BD
 \node{\B/\K}\node{\B}\arrow{w,,A}\node{\K}\arrow{w,,V}
\ED\;,
\EE
with $A_{lm}$, we obtain the extension
\BE{CBK Alm}
\BD
 \node{A_{l,m+1}}\node{\B\xT A_{lm}}\arrow{w,,A}
  \node{A_{l+1,m}}\arrow{w,,V}
\ED\;.
\EE
Exactness of \eqref{CBK Alm} is, at least in the complex case,
a well known feature of the maximal tensor product of
$C^{\ast}$-algebras, see \cite{Fack}.

In view of \eqref{KB A=0}--\eqref{KtB A=0},
commutative diagram \eqref{LES ABC}, whose rows
are formed by the long exact sequences associated
with extension \eqref{CBK Alm}, splits into a sequence of
commutative squares with connecting homomorphisms being isomorphisms
\BE{CBK A n}
\BD\dgARROWLENGTH=2.6em
  \node{K_n(A_{l+1,m})}\arrow{s,l}{\ep_n^{l+1,m}}
   \node{K_{n+1}(A_{l,m+1})}\arrow{w,tb}{\dd_{n+1}}{\sim}
     \arrow{s,l}{\ep_{n+1}^{l,m+1}}\\
  \node{K_n^\t(A_{l+1,m})}
   \node{K_{n+1}^\t(A_{l,m+1})}\arrow{w,tb}{\dd_{n+1}^\t}{\sim}
\ED\;.
\EE
This demonstrates that $A_{l+1,m}$ is $K_n$-stable if
and only if $A_{l,m+1}$ is $K_{n+1}$-stable.
Since every Banach algebra is $K_0$-stable, we deduce that
$A_{lm}$ is $K_n$-stable if $-l\leq n\leq0$.

Note that
\[
F_{l0}=\K(H)^{\xT l}
 =\K\bigl(H^{\ti\T l}\bigr)\simeq F_{10},
  \qquad(l\in\N),
\]
where $H^{\ti\T l}$ is the completion of
the pre-Hilbert space $H^{\T_F l}$.
In particular, for any $C^{\ast}$-algebra, one has
a $C^{\ast}$-algebra isomorphism
\[
A_{10}\simeq A_{l0}\qquad(l>0).
\]
In the theory of $C^{\ast}$-algebras, an algebra $A$
is said to be \emph{stable} if it is isomorphic to $A_{10}$.
The terminology reflects the fact $\K\xT A$ is
the (unique) $C^{\ast}$-algebra completion of
$M_\infty(A)$, as well as the fact that
$\K\xT A\simeq A$ if $A$ is a stable $C^{\ast}$-algebra.
The above argument demonstrates that stable
$C^{\ast}$-algebras are $K_n$-stable for $n\leq0$.

We shall adapt the above argument to general Banach algebras
by replacing $\xT$ with $\pT$ in the definition of
the family $F_{lm}$,
\BE{F^lm}
\hF_{lm}:=\K^{\pT l}\pT(\B/\K)^{\pT m}\qquad(l,m\in\N),
\EE
and, given a Banach algebra $A$, in the definition of
the associated 2-parameter family $A_{lm}$,
\BE{A^lm}
\hA_{lm}:=\hF_{lm}\pT A,
\EE
remembering that the projective tensor product is formed in
the category of Banach spaces over $F$.
As an extension of Banach spaces,
\eqref{CBK} is \emph{pure}
(see Appendix \ref{Ap BAI}),
therefore projective tensor product
with $\hA_{lm}$ preserves its exactness
and we obtain the extension
\BE{CBK A^lm}
\BD
 \node{\hA_{l,m+1}}\node{\B\pT\hA_{lm}}\arrow{w,,A}
  \node{\hA_{l+1,m}}\arrow{w,,V}
\ED\;.
\EE
By invoking \eqref{KB A=0}--\eqref{KtB A=0} again,
we conclude precisely as we did for $C^{\ast}$-algebras
that $\hA_{lm}$ is $K_n$-stable if $-l\leq n\leq0$.

For a given orthogonal projection $p\in\K$ of rank 1,
the correspondence
\[
\K(H)\LA\K(H)^{\pT l},\qquad
 c\LMT p^{\pT(l-1)}\pT c,
\]
induces a Banach algebra homomorphism $\k_l\Cl\hF_{10}\to\hF_{l0}$.
For $L=p(H)^{\T_F(l-1)}$ let $L^\perp$ be the orthogonal complement
of $L$ in $H^{\ti\T(l-1)}$. The obvious identification
of $L\T_F H$ with $H$ combined with an arbitrary Hilbert
space isomorphism $L^\perp\ti\T H\simeq H$, induces
an isomorphism of $C^{\ast}$-algebras
$\K\bigl(H^{\ti\T l}\bigr)\simeq M_2(\K(H))$.
When composed with the tensor-product-of-operators map
\[
\K(H)^{\pT l}\LHRA\K\bigl(H^{\ti\T l}\bigr),
\]
it produces a homomorphism of Banach algebras
\[
\nu_l\Cl\K^{\pT l}\LA M_2(\K(H))
\]
and the composite $\nu_l\c\k_l$ coincides with
$\i_2\Cl\K\HRA M_2(\K)$.
Setting $\k=\k_l\T_F\id_A$ and $\nu=\nu_l\T_F\id_A$ provides
then a similar factorization of
$\i_2\Cl\hA_{10}\HRA M_2(\hA_{10})$
through $\hA_{l0}$. This demonstrates that,
for any Banach algebra $A$
and any $l>0$, the algebra $\hA_{10}$ is a stable retract of
$\hA_{l0}$. In particular, $\hA_{10}$ is $K_n$-stable for
any $n\leq0$.

\medskip
We arrive at the following theorem which is a generalization
of the main result of \cite{KaroubiSLN}. We emphasize that
it holds for real and complex Banach algebras alike.

\begin{theorem}\label{th:n<0}
Any Banach algebra which is a stable retract of
$\K\xT A$, for some $C^{\ast}$-algebra $A$, or of $\K\pT A$,
for some Banach algebra $A$, is $K_n$-stable for $n\leq0$.
\end{theorem}

\begin{remark}
The bilinear pairing
\[
\K\x(\K\xT A) \LA \K\xT(\K\xT A),\qquad(c,\a)\LMT c\T\a,
\]
gives rise to a homomorphism of Banach algebras
\[
\K\pT(\K\xT A) \LA \K\xT(\K\xT A).
\]
Composition with $\i\pT\id_{\K\xT A}$, where $\i\Cl F\HRA\K$
is the inclusion defined in \eqref{F->K},
is a homomorphism
\[
\K\xT A\LHRA \K\xT(\K\xT A)\simeq(\K\xT\K)\xT A
\]
which is isomorphic to the inclusion $\K\xT A\HRA M_2(\K\xT A)$.
In particular, $\K\pT A$ is a stable retract of $\K\pT(\K\xT A)$,
thus a separate argument used by us to demonstrate
$K_n$-stability of $\K\xT A$ is, in fact, redundant.
\end{remark}

\bigskip
\section{$K$-theory of $\Kc$-rings}\label{s:Kst C}

\smallskip
In 1970-ties one of us \cite{KaroubiSLN}
discovered that algebraic K-theory of unital rings
and topological K-theory of Banach algebras both possess
$\Z$-graded multiplicative structures that are
associative, graded-commutative, and compatible
with the comparison map
\BE{CM}
\ep_{\ast}\Cl K_{\ast}(A)\LA K_{\ast}^\t(A),
\EE
where
\[
K_{\ast}(A)=\bigoplus_{n\in\Z}K_n(A)
 \AND
  K_{\ast}^\t(A):=\bigoplus_{n\in\Z}K_n^\t(A),
\]
and $A$ denotes a Banach algebra.
That early observation provides a convenient
starting point for this chapter.
For any unital ring $R$, the tensor product map
\BE{tens}
\TZ\Cl R\x R\to R\TZ R,\qquad(r,r')\mapsto r\T r',
\EE
is a \emph{bimultiplicative} pairing in the sense of
Appendix \ref{Ap MSK} and therefore induces
a homomorphism of $\Z$-graded abelian groups
\BE{t K}
K_{\ast}(R)\TZ K_{\ast}(R)\to K_{\ast}(R\TZ R).
\EE
It follows that any unital ring homomorphism
\BE{mu}
\mu:R\TZ R\to R
\EE
induces a $\Z$-graded (nonassociative) ring structure on $K_{\ast}(R)$.

If the multiplication map \eqref{mu} is $\a$-associative
for a certain ring automorphism $\a\in\Aut(R)$, i.e., if 
\[
\mu\c(\id_{R}\T\mu)=\a\c\mu(\mu\T\id_{R}),
\]
 then the induced multiplication on $K_{\ast}(R)$ is
$\a_{*}$-associative,
where $\a_{*}$ denotes the induced automorphism of $K_{\ast}(R)$.

Similarly, if $\mu$ is $\b$-commutative
for a certain ring automorphism $\b\in\Aut(R)$, i.e., if 
\[
\mu\c\tau=\b\c\mu,
\]
where $\tau$ transposes factors of $R\TZ R$, then the induced
multiplication on $K_{\ast}(R)$ is graded $\b_{*}$-commutative,
\[
vu=(-1)^{\ti u\ti v}\b_{\ast}(uv),
\]
where $\ti w$ denotes the \emph{parity} of the degree
of an element $w\in K_n(R)$.

In a similar vein, any unital ring homomorphism 
\BE{lambda}
\l:R\TZ S\to S
\EE
induces a graded map 
\[
K_{\ast}(R)\TZ K_{\ast}(S) \LA K_{\ast}(S)\label{K_*-mod}.
\]
If $R$ itself is equipped with a multiplication \eqref{mu},
and if there exists an automorphism $\g\in\Aut(S)$ such that 
\[
\l\c(\id_{R}\T\l)=\g\c\l\c(\mu\T\id_{S}),
\]
then the product maps induced on the corresponding
$\Z$-graded $K$-groups by $\l$ and $\mu$ satisfy
a similar associativity identity twisted by $\g_{*}$.

\medskip
Some of the most interesting ``multiplications'' $\mu$
and ``actions'' $\l$ involve nonunital rings.
Extension of the above multiplicative structures
to $K$-theory of nonunital rings is straightforward
provided each of the following rings,
\[
R\TZ R,\qquad R\TZ S,\qquad
 R\TZ R\TZ R,\AND R\TZ R\TZ S,
\]
satisfies excision in algebraic $K$-theory.

The problem how to characterize rings that
satisfy excision in rational algebraic
$K$-theory was solved in a pair of articles
\cite{Wodzicki.Annals} and \cite{Suslin-Wodzicki}: 
\BE{ChEx}
\parbox[s]{3.1in}{\textit{a ring $R$ satisfies
 excision in rational algebraic $K$-theory if
 and only if $R\TZ\Q$ is an $H$-unital
 $\Q$-algebra.}}
\EE

\emph{$H$-unitality}
 	was introduced in \cite{Wodzicki.CR}
	where it was shown to characterize
	algebras satisfying excision in
	Hochschild and cyclic homology.
For $\Q$-algebras, excision in $K$-theory and
in rational $K$-theory are equivalent, hence
\eqref{ChEx}
provides also a complete characterization of $\Q$-algebras
that satisfy excision in algebraic $K$-theory
\cite[Theorem B and the remark preceding it]{Suslin-Wodzicki}.
The validity of the above theorem is closely related to
the following surprisingly delicate fact
\cite[Theorem 7.10]{Suslin-Wodzicki}
\BE{H-un tens}
\parbox[s]{2.5in}{
\textit{the category of $H$-unital $k$-algebras
 over an arbitrary ground ring $k$ is closed
 under $\T_k\,$}.}
\EE

This result is particularly useful when extending
the multiplicative structures of $K$-theory to
$H$-unital $\Q$-algebras. It guarantees, in particular,
that the ``multiplication'' and ``action'' induced
in $K$-theory by $\mu$ and $\l$ satisfy twisted
associativity and commutativity properties,
mirroring the ones satisfied by $\mu$ and $\l$
themselves.

\medskip
Remarkably, $H$-unitality is a relatively
frequent phenomenon among rings
originating in Operator Theory and
Functional, as well as Global Analysis.
For example, any Banach algebra with bounded
left approximate identity satisfies
the Triple Factorization Property ($\Phi$)
introduced in \cite{Wodzicki.PNAS}
(cf.~\cite[Proposition 10.2]{Suslin-Wodzicki}).
Any ring $R$ with this property satisfies excision
in algebraic $K$-theory \cite[Theorem C]{Suslin-Wodzicki}.
Such a ring is also a flat right module over an
arbitrary unital ring $A$ that contains $R$ as
a right ideal (cf.\ \cite[Proposition 4]{Wodzicki.PNAS}
where this property of $R$ is called \emph{right
universal flatness}).
Similarly for Banach algebras with bounded right
approximate identity: they are left universally flat.

A universally flat ring is $H$-unital in the
category of $k$-algebras for any ground ring $k$
and any \emph{flat} $k$-algebra structure on $R$,
cf.\ \cite[Theorem 1]{Wodzicki.PNAS}. In particular,
Banach algebras with bounded one-sided approximate
identities satisfy excision in algebraic $K$-theory
and are $H$-unital as $k$-algebras for every flat
$k$-algebra structure on them.
Projective tensor product of Banach algebras with
bounded left approximate identity has a
bounded left approximate identity itself.
Finally, tensor products over $\Z$ of Banach algebras
with bounded approximate identity are $H$-unital
as $\Z$-algebras, and therefore also as $\Q$-algebras,
according to Theorem \eqref{H-un tens} quoted above.

\smallskip
This litany of remarkable algebraic and homological
properties means that the mutiplicative structures
discussed at the beginning of this chapter
extend to any Banach algebra with one-sided bounded
approximate identity.
This includes every $C^{\ast}$-algebra as well as every
closed one-sided ideal in a $C^{\ast}$-algebra.
It also includes $\K\pT A$ where
$A$ is an arbitrary Banach algebra
with one-sided bounded approximate identity
(note that $\K$ itself possesses a two-sided
approximate identity bounded by 1).

\smallskip
Let us record one more favorable property
of $H$-unital algebras. If a $\Q$ algebra $R$
is $H$-unital, so is the matrix algebra $M_l(R)$
\cite[Corollary 9.8]{Wodzicki.Annals} (this is,
of course, a special case of Theorem \eqref{H-un tens}),
and therefore satisfies excision in algebraic
$K$-theory. This implies that the stabilization
map of \eqref{i_l} induces an isomorphism in
algebraic $K$-theory. The argument
utilised in the proof of Corollary 9.10 in \cite{Wodzicki.Annals}
then shows that conjugation by any element
$a\in\GL_1(A)$, for any unital ring $A$ containing
$R$ as a two-sided ideal, acts trivially on $K_{\ast}(R)$.

\smallskip
If $R$ and $S$ are Banach algebras and if $\l$, $\mu$ and
the constraint maps, $\a$, $\b$, and $\g$, are continuous,
then $K_{\ast}^\t(R)$ becomes
a $\Z$-graded, $\a_{\ast}$-associative, graded $\b_{\ast}$-commutative ring,
and $K_{\ast}^\t(S)$ becomes a $\Z$-graded, $\g_{\ast}$-associative
$K_{\ast}^\t(R)$-module. Since excision holds in topological
$K$-theory for all Banach algebras without exception,
neither $R$ nor $S$ are subjected to further
restrictions similar to $H$-unitality.
For Banach algebras which are $H$-unital
as $\Q$-algebras, the multiplicative structures
in algebraic and topological $K$-theory are compatible
with the comparison map, i.e.,
the following diagrams,
\BE{comp mult}
\BD
 \node{K_{\ast}(R)\x K_{\ast}(R)}\arrow{e,t}{\mu_{\ast}}\arrow{s,l}{\ep_{\ast}\x\ep_{\ast}}
  \node{K_{\ast}(R)}\arrow{s,r}{\ep_{\ast}}\\
 \node{K_{\ast}^\t(R)\x K_{\ast}^\t(R)}\arrow{e,t}{\mu_{\ast}}\node{K_{\ast}^\t(R)}
\ED
\aND
\BD
 \node{K_{\ast}(R)\x K_{\ast}(S)}\arrow{e,t}{\l_{\ast}}\arrow{s,l}{\ep_{\ast}\x\ep_{\ast}}
  \node{K_{\ast}(S)}\arrow{s,r}{\ep_{\ast}}\\
 \node{K_{\ast}^\t(R)\x K_{\ast}^\t(S)}\arrow{e,t}{\l_{\ast}}\node{K_{\ast}^\t(S)}
\ED\;,
\EE
commute.

An important example of a noncommutative ring equipped with
a multiplication \eqref{mu} is provided by the ring
$\K=\K(H)$ of compact operators on a separable,
real or complex, infinite-dimensional Hilbert space.
Let us consider the map induced by tensor product of
operators
\begin{equation}\label{TPO}
\K(H)\TZ\K(H) \LA \K(H)\T_F \K(H)\LHRA\K(H\ti{\T}H)
\end{equation}
where $H\ti\T H$ denotes the Hilbert-space completion of
the pre-Hilbert space $H\T_F H$ and $F$ denotes the field
of coefficients of $H$.

Composition of \eqref{TPO} with the isomorphism 
\begin{equation}\label{K HtH=K H}
\K(H\ti\T H)\simeq\K(H)
\end{equation}
induced by a fixed but otherwise arbitrary identification
$H\ti{\T}H\simeq H$ of Hilbert spaces, defines
a multiplication $\mu$ on $\K$ which depends
on the choice of that identification.
This multiplication is $\ad_{U}$-associative and
$\ad_{V}$-commutative for suitable bounded invertible
operators $U,V\in\B(H)$ where $\ad_{E}(X):=EXE^{-1}$.

Any orthogonal projection $p\in\K$ of rank 1 is a
``stable quasiidentity'' for this multiplication, which
means that the following diagram commutes
\BE{qid}{\dgARROWPARTS=8
\BD
\node[2]{\K}\arrow{se,t,}{\ad_{W'}}\\
\node{\K}\arrow[2]{e,t,V,5}{\i_2}\arrow{ne,t,}{\mu(p\T\;\;)}
                            \arrow{se,b,}{\mu(\;\;\T p)}
 \node[2]{M_2(\K)}\\
\node[2]{\K}\arrow{ne,b,}{\ad_{W''}}
\ED}
\EE
for suitable isomorphisms $W'$ and $W''$ between
$H$ and $H\ds H$, where $\i_2$, as usual, denotes
the stabilization map of \eqref{i_l}.

\medskip
We are ready to introduce a certain class of
rings equipped with a quasiassociative
action of $\K$, i.e., with
a ring homomorphism
\BE{KR->R}
\l\Cl\K\T_\Z R\to R
\EE
which is $\g$-associative for some $\g\in\Aut\,R$.
Note that quasiassociativity involves both $\l$
and a fixed multiplication $\mu$ on $\K$.
If \eqref{KR->R} is ``stably quasiunitary''
which means that the following diagram commutes
\BE{qun}
\begin{diagram}
\node{R}\arrow{e,t,V}{\i_2}\node{M_2(R)}\arrow{s,r}{\th}\\
\node{R}\arrow{e,t,V}{\i_2}\arrow{n,l}{\l(p\T\;\; )}
 \node{M_2(R)}
\end{diagram}
\EE
for a suitable automorphism $\th$ of $M_2(R)$,
we will call $R$ a \emph{$\K$-ring}.
More precisely, we will call it a \emph{real $\K$-ring},
if $\K=\Kr$ is the algebra of compact operators on
a real Hilbert space,
and a \emph{complex $\K$-ring}, if $\K=\Kc$
is the algebra of compact operators on
a complex Hilbert space. Alternatively,
we will be talking of $\Kr$- and $\Kc$-rings.

A Banach algebra will be called a
\emph{Banach $\K$-ring} if \eqref{KR->R}
induces a homomorphism of Banach algebras
$\K\pT R\to R$ while $\g$ and $\th$
are continuous.


The ring of the form
$R=\K\T_k A$, where $A$ is an
algebra over a subring $k\sbeq F$,
is an example of a $\K$-ring provided
$R$ is $H$-unital over $\Q$.
The projective tensor product
$R=\K\pT A$ by a Banach algebra $A$
with a one-sided bounded approximate
identity supplies an example
of a Banach $\K$-ring. Finally,
\emph{stable} $C^{\ast}$-algebras
$R=\K\xT A$ are examples of
Banach $\K$-rings among
$C^{\ast}$-algebras.

These are the obvious examples, so to speak.
Less obvious is the Calkin algebra
$\B/\K$ and its tensor products in any
of the above three senses.
Another example of a unital Banach $\K$-ring
is suppplied by the algebra defined as the
fibered square of $\B$ over $\B/\K$
\BE{B2}
\B^{(2)}:=\{(a,b)\in\B\x\B\mid a-b\in\K\}.
\EE
The $C^{\ast}$-algebra defined in \eqref{B2}
fits into an extension
\BE{BB2K}
\BD
 \node{\B}\node{\B^{(2)}}\arrow{w,t,A}{p}\node{\K}\arrow{w,t,V}{i}
\ED\;,
\EE
where
\BE{KB2}
i\Cl\K\LHRA \B^{(2)},\qquad c\LMT(c,0),\qquad(c\in\K).
\EE
The long exact sequence in either algebraic or
topological $K$-theory demonstrates
independently of Theorem \ref{th:PTC}
that \eqref{KB2} induces isomorphisms
in $K$-theory
\[
K_{\ast}(\K)\simeq K_{\ast}\bigl(\B^{(2)}\bigr)\AND
 K_{\ast}^\t(\K)\simeq K_{\ast}^\t\bigl(\B^{(2)}\bigr).
\]

\medskip
Noting that the group $\GL_1(\B)=\GL_1(\B(H))$
acts trivially on $K_{\ast}(\K)$, we infer
that the multiplication induced by $\mu$
on $K_{\ast}(\K)$ is associative, graded commutative,
and does not depend on the isomorphism in
\eqref{K HtH=K H} induced by the
chosen identification $H\ti{\T}H\simeq H$.

\smallskip
For a one dimensional subspace $H'\subset H$,
the map induced by tensor product of operators
\[
\K(H')\x\K(H')\LA\K(H'\T_FH')
\]
coincides with multiplication in the ground field.
Moreover, the embedding of $F$ into $\K(H)$
which corresponds to the inclusion $H'\HRA H$, cf.\ \eqref{F->K},
is a homomorphism of the ground field
into $(\K,+,\mu)$ provided we choose an
isomorphism $H\ti\T H\simeq H$ which identifies
the one-dimensional subspace $H'\T H'\subset H\ti\T H$
with $H'\subset H$.
The induced map $\i_{\ast}\Cl K_{\ast}(F)\to K_{\ast}(\K)$ in this case
is a homomorphism of graded rings and simultaneously
a homomorphism of graded $K_{\ast}(F)$-modules.

\smallskip
In conclusion, the $\Z$-graded algebraic $K$-group
of any $\K$-ring $R$ which is $H$-unital as a $\Q$-algebra,
becomes a graded $\g_{\ast}$-associative module over $K_{\ast}(\K)$,
and the class $[p]$ of any rank 1 projection, which
is a generator of $K_0(\K)\simeq\Z$, acts on $K_{\ast}(R)$
via automorphism $\th_{\ast}$
\[
[p]w=\th_{\ast}(w),\qquad(w\in K_{\ast}(R)).
\]
For $R=\K\T_k A$, as well as for
$\K\pT A$ and $\K\xT A$,
constraints $\g$ and $\th$ can be realized as
inner automorphisms of unital rings that contain $R$
and, respectively, $M_2(R)$, as two-sided ideals.
Accordingly, $\g_{\ast}=\id$ and $\th_{\ast}=\id$, and
$K_{\ast}(R)$ is in those cases a strictly associative
and unitary $K_{\ast}(\K)$-module.

\bigskip
The $\Z$-graded \emph{topological} $K$-theory ring
of $\C$ is canonically isomorphic to the ring of Laurent
polynomials $\Z[u,u^{-1}]$ where $u\in K_2^\t(\C)$
corresponds to the generator of $\ti K_\C\bigl(S^2\bigr)$,
usually chosen to be the class of the line bundle $\mathcal O(1)$
on $P^1(\C)\simeq S^2$ \cite[Theorem 2.4.9]{AtiyahBook}. Accordingly,
the $\Z$-graded topological $K$-theory ring of any complex
Banach algebra $A$ becomes a graded unitary module over
$K_{\ast}^\t(\C)\simeq\Z[u,u^{-1}]$. Multiplication by $u$
defines canonical isomorphisms
$K_n^\t(A)\simeq K_{n+2}^\t(A)$, $n\in\Z$,
multiplication by $u^{-1}$ provides the inverse isomorphisms.
This is, of course, Bott periodicity in topological
$K$-theory of complex Banach algebras, seen through
the multiplicative structures of topological $K$-theory.

Any Banach algebra homomorphism $f\Cl A\to B$
induces a homomorphism $f_{\ast}\Cl K_{\ast}^\t(A)\to K_{\ast}^\t(B)$
of graded modules over $K_{\ast}^\t(\C)\simeq\Z[u,u^{-1}]$.
In particular, $f_{\ast}$ is an isomorphism if and only if
$f_n$ is an isomorphism for at least a single even
and a single odd degree $n$.
This is so for the inclusion $\i\Cl F\HRA\K$
defined in \eqref{F->K} since
\[
\i_0\Cl K_0(\C)\simeq K_0(\Kc)\AND
K_1^\t(\C)=K_1^\t(\Kc)=0.
\]
Hence $\i_{\ast}$ is an isomorphism of
graded $K_{\ast}^\t(\C)$-modules and,
being also a ring homomorphism,
is an isomorphism of graded rings.

In \cite{KaroubiSLN} and \cite{KaroubiJ.op.theory}
we demonstrated that the comparison maps
\[
\ep_{\pm2}\Cl K_{\pm2}(\Kc)\LA K_{\pm2}^\t(\Kc)
\]
are surjective.
Note that existence of $u_{-2}$ was already established
in Section \ref{s:CF} where we proved
that $K_n(\K)\simeq K_n^\t(\K)$ for all $n\leq0$,
in real and complex cases alike.

Since $\ep_0$ is the identity map and
the comparison map is a homomorphism of
graded rings,
there must exist elements $u_{\pm2}\in K_{\pm2}(\Kc)$ such that
$u_{-2}u_2=1\in K_0(\Kc)\simeq\Z$.
The graded ring $K_{\ast}(\Kc)$ thus possesses an invertible
element of degree 2. By invoking Theorem \ref{th:n<0},
which says that $\ep_{\ast}$ is an isomorphism in degrees $\leq0$
we arrive at the following result.

\begin{theorem}[\cite{Suslin-Wodzicki}, \cite{Wodzicki.ECM}]\label{th:Kc comp}
The comparison map $\ep_{\ast}\Cl K_{\ast}(\K(H))\to K_{\ast}^\t(\K(H))$,
where $H$ denotes a complex infinite dimensional separable
Hilbert space, is an isomorphism of $\Z$-graded unital rings.
\end{theorem}

\medskip
The following result in essence was predicted by
the first author, and was proved by the second
\cite{Wodzicki.ECM}. It is an immediate
corollary of Theorem \ref{th:Kc comp} and
the fact that $u_2$ acts on
$K_{\ast}(R)$, for any $\Kc$-ring which is $H$-unital
as a $\Q$-algebra, via an isomorphism of degree 2.

\begin{theorem}[\cite{Wodzicki.ECM}]\label{th:PTC}
Let $R$ be any $\Kc$-ring which is $H$-unital as
a $\Q$-algebra. Then the algebraic
$K$-groups are periodic with period 2, the periodicity
isomorphism realized as multiplication by
$u_2\in K_2(\Kc)$.
\end{theorem}

Let us reiterate that besides the rings
of the form $R=\Kc\T_k A$ which are $H$-unital
over $\Q$, the above periodicity theorem
applies to any Banach $\Kc$-ring, e.g.,
to $R=\Kc\pT A$, where $A$ is any real or complex Banach
algebra with bounded one-sided approximate identity,
and to $\Kc\xT A$ where $A$ is any $C^{\ast}$-algebra.

For a Banach $\Kc$-ring,
the comparison map $\ep_{\ast}\Cl K_{\ast}(R)\to K_{\ast}^\t(R)$
is a homomorphism of graded $\Z[u,u^{-1}]$-modules
which, in view of Theorem \ref{th:n<0},
is an isomorphism in degrees $\leq0$.
This leads to the following result.

\begin{theorem}\label{th:KCst} \ 
Any stable retract of a Banach $\Kc$-ring $R$ that
is $H$-unital over $\Q$ is $K_n$-stable for every $n\in\Z$.
In other words, the comparison map
$\ep_{\ast}\Cl K_{\ast}(R)\to K_{\ast}^\t(R)$
is an isomorphism of $\Z$-graded groups.
\end{theorem}

\begin{remark}Theorem \ref{th:KCst} contains as special cases
various formulations of the so-called Karoubi’s Conjecture:
the $C^{\ast}$-algebra form proved in
\cite[Theorem 10.9]{Suslin-Wodzicki}
and the Banach algebra form proved in \cite{Wodzicki.ECM}.
The proof in \cite{Suslin-Wodzicki} does not exploit
multiplicative structures on $\Z$-graded $K$-groups.
Instead, it utilizes the characterization of homotopy invariant
functors from the category of $C^{\ast}$-algebras to the category
of abelian groups due to Higson \cite{Higson}, whose work
extended earlier results and techniques of Kasparov and Cuntz.
(In recent years that was further extended to functors of
bornological algebras \cite[Section 3.2]{CMR}.)
It also invokes the fact that any complex stable $C^{\ast}$-algebra
is $K_1$-stable (the proof of that statement is elementary).
\end{remark}

\bigskip
\section{$K$-theory of $\Kr$-rings}\label{s:Kst R}

\smallskip
Let $H$ denote an arbitrary \emph{real} Hilbert space.
The diagonal embedding
\BE{D2}
\D_2\Cl\K(H)\LHRA\K(H\ds H),\qquad c\LMT\BM{cc}c&0\\0&c\EM,
\EE
factorizes
\BE{D=rc}
\BD\dgARROWLENGTH=0.7em\dgARROWPARTS=8
\node{\K(H)}\arrow[2]{e,t,V,5}{\D_2}\arrow{sse,b,V}{\g}
 \node[2]{\K(H\ds H)}\\
 \\
\node[2]{\K(H\T_\R\C)}\arrow{nne,b,V}{\rho}
\ED\;,
\EE
where $\g$ is the complexification map, $c\mapsto c\T_\R\id_\C$,
while $\rho$ is the inclusion $\Kc(H\T_\R\C)\HRA\Kr(H\T_\R\C)$.
Out of these three ring homomorphisms only $\g$ is also
a homomorphism of the corresponding $\mu$-multiplications,
hence it induces a homomorphism of graded $K$-theory rings, the other
two induce homomorphisms of $K_{\ast}(\K(H))$-modules in algebraic
$K$-theory, and of $K_{\ast}^\t(\K(H))$-modules---in
topological $K$-theory.

The homomorphism defined in \eqref{D2}
induces multiplication by 2
in algebraic and in topological $K$-theory alike.
Since $\ep_{\ast}\Cl K_{\ast}(\K(H\T_\R\C))\to K_{\ast}^\t(\K(H\T_\R\C))$
is an isomorphism in view of Theorem \ref{th:Kc comp},
the kernel and the cokernel of the comparison map
\BE{epR}
\ep_{\ast}\Cl K_{\ast}(\K(H))\to K_{\ast}^\t(\K(H))
\EE
are abelian groups of exponent 2.

\begin{theorem}\label{th:Kr comp}
The comparison map $\ep_{\ast}\Cl K_{\ast}(\K(H))\to K_{\ast}^\t(\K(H))$,
where $H$ denotes a real infinite dimensional separable
Hilbert space, is an isomorphism of $\Z$-graded unital rings.
\end{theorem}

By Theorem \ref{th:n<0}, the comparison map
is an isomorphism in degrees $n\leq0$.
It remains to find an invertible element
of degree not equal 0 in the ring
$K_{\ast}(\K(H))$.  Such an element is sent
by \eqref{epR} to an invertible
element in $K_{\ast}^\t(\K(H))$. Recalling that
inclusion of a one-dimensional subspace into $H$
induces an inclusion of $\R$ into $\K(H)$,
cf.~\eqref{F->K}, we turn our attention
to the ring structure of $K_{\ast}^\t(\R)$,
\BE{KR}
K_{\ast}^\t(\R)\;\simeq\;
 \Z[\eta,w,v,v^{-1}]/(\eta^3,2\eta,\eta w,w^2-4v)\,.
\EE
Here $\eta$, $w$ and $v$ have degrees 1, 4
and, respectively, 8 \cite[p.\,157]{MKbook}.
In particular, $\{\pm v^n\mid n\in\Z\}$ is its
group of invertible elements; the degrees of
such elements are multiples of 8.

For a one-dimensional Hilbert space, diagram
\eqref{D=rc} becomes
\BE{D=rc1}
\BD\dgARROWPARTS=8
\node{\R}\arrow[2]{e,t,V,5}{\D_2}\arrow{se,b,V}{\g}
 \node[2]{M_2(\R)}\\
\node[2]{\C}\arrow{ne,b,V}{\rho}
\ED\;,
\EE
where $\g$ is the inclusion map, and $\rho$ sends $a+ib$
to $\BM{rr}a&-b\\b&a\EM$.

The induced ring homomorphism
$\g_{\ast}^\t\Cl K_{\ast}^\t(\R)\to K_{\ast}^\t(\C)$
sends $w$ to $2u^2$ and $v$ to $u^4$,
while
\[
\rho_{\ast}^\t(u^{4l+m})=\begin{cases}
     2v^l & \quad\text{if}\quad m=0\\
 \eta^2v^l & \quad\text{if}\quad m=1\\
    w v^l & \quad\text{if}\quad m=2\\
    0     & \quad\text{if}\quad m=3\end{cases}
\]
cf.\ \cite[pp.\,157--158]{MKbook}.

\smallskip
Let us consider the algebraic and topological $K$-groups
of $\Kr=\K(H)$ with coefficients $\Z/2^l\Z$,
cf. \cite{Araki-Toda}, \cite{Browder}, \cite{Neisendorfer}.
They fit into the commutative diagram
\[
\BD\dgARROWLENGTH=1.0\dgARROWLENGTH
\node{K_8\bigl(\Kr;\Z/2^l\Z\bigr)} \arrow[2]{s,l}{\bep}
 \node{K_8(\Kr)}\arrow{w,t}{r}\arrow[2]{s,l}{\ep_\R}
   \node[2]{K_8(\Kr)} \arrow[2]{w,t}{\x2^l} \arrow[2]{s,l}{\ep_\R} \arrow{sw,b}{\g}
\\
\node[3]{K_8(\Kc)} \arrow{nw,b}{2^{l-1}\rho} \arrow[2]{s,r,1}{\ep_\C}
\\
\node{K_8^\t(\Kr;\Z/2^l\Z)}
 \node{K_8^\t(\Kr)}\arrow{w,t}{r^\t}
   \node{}\arrow{w,t}{\x2^l}
      \node{K_8^\t(\Kr)} \arrow{w,,-} \arrow{sw,b}{\g^\t}
\\
\node[3]{K_8^\t(\Kc)} \arrow{nw,b}{2^{l-1}\rho^\t}
\ED
\]
whose rows are portions of the long exact sequence relating
the homotopy groups of a space to the homotopy groups
with coefficients $\Z/2^l\Z$.
The \emph{reduction} mod $2^l$ maps, denoted $r$,
are induced by the \emph{pinching} maps
$P^n\bigl(2^l\bigr)\to S^n$
from \emph{Peterson} spaces
\[
P^n(d)=S^{n-2}\w P^2(d),\qquad(n\geq2),
\]
which are the cofibers of maps between spheres
$S^{n-1}\to S^{n-1}$ of degree $d$.
The vertical arrows are the appropriate comparison maps.

The comparison map for $\Kc=\K(H\T_\R\C)$, denoted $\ep_\C$,
is an isomorphism by Theorem \ref{th:KCst} and it sends element
$u_2^4\in K_8(\Kc)$ to $u^4\in K_8^\t(\Kc)$.
In view of
\[
\GL_\infty(\R)\HRA(1+\K(H))^{\ast}
 \AND
\GL_\infty(\C)\HRA(1+\K(H\T_\R\C))^{\ast}
\]
being homotopy equivalences \cite{Palais},
the vertical arrows in the commutative diagram
\[
\BD
\node{K_{\ast}^\t(\C)} \arrow{e,t}{\rho^\t} \arrow{s,lr}\i\simeq
   \node{K_{\ast}^\t(\R)} \arrow{s,lr}\i\simeq \\
\node{K_{\ast}^\t(\Kc))} \arrow{e,t}{\rho^\t}
      \node{K_{\ast}^\t(\Kr)}
\ED
\]
are isomorphisms and we infer that
\[
\LL(\bep\c r\c2^{l-1}\rho\RR)\bigl(u_2^4\bigr)
 = \LL(r^\t\c2^{l-1}\rho^\t\RR)\bigl(u^4\bigr)
  = r^\t\bigl(2^lv\bigr) = 0.
\]

If $\bep$ is injective, then exactness
of the top row implies that
\[
2^{l-1}\rho\bigl(u_2^4\bigr)=2^lv_8
\]
for some element $v_8\in K_8(\Kr)$.
Accordingly,
\[
2^lv=2^{l-1}\rho^\t\bigl(u^4\bigr)
 =\ep_\C\bigl(u_2^4\bigr)=2^l\ep_\R(v_8)
\]
in $K_8^\t(\Kr)\simeq\Z$. As the latter group
is torsion free, it follows that $v=\ep_\R(v_8)$,
i.e.,\ the comparison map $K_8(\Kr)\to K_8^\t(\Kr)$
is surjective. The $\Z$-graded comparison map
\eqref{epR}, being a unital ring homomorphism
and an isomorphism in degrees less than or equal 0,
is then seen to be an isomorphism
in all degrees, with multiplication by $v_8$
providing 8-periodicity isomorphisms
$K_n(\Kr)\simeq K_{n+8}(\Kr)$.

The proof of Theorem \ref{th:Kr comp} will be complete
if we show that $\bep$ is injective for at least
one $l>0$. This follows, for example, from a celebrated
result of Suslin about $K$-theory of $\R$ and $\C$
with finite coefficients. According to
\cite[Corollary 4.8]{Suslin}, one has
\[
\bep_n\Cl K_n(\R;\Z/d\Z)\xrightarrow{\quad\sim\quad} K_n^\t(\R;\Z/d\Z)
 \qquad(n>0,\,d>1).
\]
Below we offer a proof that
does not rely on Suslin's Theorem. 

\begin{lemma}
The $\Z$-graded comparison map
\BE{ep16}
\bep_{\ast}\Cl K_{\ast}(\Kr;\Zsixt) \LA K_{\ast}^\t(\Kr;\Zsixt)
\EE
is an isomorphism.
\end{lemma}

\begin{proof}
There exists an element of order 16 in
$\pi_8\LL(\Om^\infty S^\infty;\Zsixt\RR)$
which under the sequence of obvious maps
\[
\Om^\infty S^\infty\sim \textup{B}\Sigma_\infty^+ \LA \BGL_\infty(\Z)^+
 \LA \BGL_\infty(\R)^+ \LA \BGL_\infty^\t(\R)
\]
is sent to the generator of $K_8^\t(\R;\Zsixt)\simeq\Zsixt$,
cf.\ e.g.,\ \cite[Theorem 4.1]{Browder}. It follows that
$\bepR$, the left comparison map in the commutative diagram
\[
\BD
\node{K_8(\R;\Zsixt)}\arrow{e,t}\i\arrow{s,r,A}{\bepR}
 \node{K_8(\Kr;\Zsixt)}\arrow{s,r}{\bepK}\\
\node{K_8^\t(\R;\Zsixt)}\arrow{e,tb}\simeq\i
 \node{K_8^\t(\Kr;\Zsixt)}
\ED\;,
\]
is surjective. Since the bottom arrow is an isomorphism by
the result of Palais mentioned above \cite{Palais}, also $\bepK$
is surjective.  Let us denote by $\bv_8$ any element of
$K_8(\Kr;\Zsixt)$ which is mapped by $\bepK$
to the generator of $K_8^\t(\Kr;\Zsixt)$.

For any $d\neq2\mod4$, there exists a coproduct map
$P^4(d)\to P^2(d)\w P^2(d)$ which, by using suspension,
generates a sequence of compatible coproduct maps
on Peterson spaces
\BE{coprod}
P^{m+n}(d) \LA P^m(d)\w P^n(d)\qquad(m,n\geq2).
\EE
The coproduct maps are known to be coassociative for
$d$ \emph{odd} and $d\neq\pm3\mod9$. This fact is relatively
easy to see when $d$ is not divisible by 3. The case 
when $9\vert d$ is significantly more delicate, cf.,
for example, \cite[Theorem 25.1]{Neisendorfer}.
The multiplications on the Moore spectrum $M_d$ 
which can be considered stable Spanier-Whitehead
duals of maps \eqref{coprod}, were studied by Araki–Toda
\cite{Araki-Toda} and Oka \cite{Oka}.
The unital multiplications exist if and only if $d\neq2$ mod 4
and, according to \cite[Theorem 2(d)]{Oka},
they are associative if and only if $d\neq\pm3$ mod 9.  

Leaving aside the question as to the exact extent
of coassociativity of the coproducts in \eqref{coprod},
it is easy to see that the coproducts exhibit what we would
like to call \emph{limited coassociativity}.
In the context of the corresponding product structures
on $K_{\ast}(\ ;\Z/d\Z)$, this means that
\[
(\a\b)\g=\a(\b\g)
\]
if at least one of the three elements is the reduction mod $d$
of an element in the \emph{integral} $K$-group, i.e.,\ belongs to
the image of the reduction map
\BE{red}
r\Cl K_{\ast}(\ )\LA K_{\ast}(\ ;\Z/d\Z).
\EE
This has been known already to Araki and Toda \cite{Araki-Toda}
who refer to it as \emph{quasi-associativity}.
As we shall see, this is sufficient for our purposes.

The comparison mod 16 map \eqref{ep16} is a homomorphism of $\Z$-graded
binary rings, as are the reduction mod 16 maps.
If we denote by $\bv_{-8}$ the reduction of
$v_{-8}\in K_{-8}(\Kr)$, and by $\bv$ the reduction of $v\in K_8^\t(\Kr)$,
then
\[
\bep_0(\bv_8\bv_{-8})=\bep_8(\bv_8)\bep_{-8}(\bv_{-8})=
 \bv\bv^{-1}=r(vv^{-1})=r(1)=1
\]
in $K_0^\t(\Kr;\Zsixt)\simeq\Zsixt$.
Since $\bep_n$ is an isomorphism for $n\leq0$, one has
$\bv_{-8}=\bv_8^{-1}$.
Limited associativity of the product structure
in $K_{\ast}(\Kr;\Zsixt)$
combined with the fact that $\bv_{-8}=r(v_{-8})$,
implies that multiplication by $\bv_{-8}$, viewed
as an operation on the $\Z$-graded abelian group
$K_{\ast}(\Kr;\Zsixt)$, is the inverse to the operation
of multiplication by $\bv_8$. In other words,
\eqref{ep16} is an isomorphism of graded modules
over the ring of Laurent polynomials in one variable
of degree 8.
\end{proof}

\begin{remark}
Note that Theorem \ref{th:Kr comp} implies that
$\bv_8$ is the reduction mod 16 of a generator of
$K_8(\Kr)\simeq K_8^\t(\Kr)\simeq\Z$, unlike
the element in $K_8(\R;\Zsixt)$
which allowed us to prove the existence
of $\bv_8$: according to
\cite[Theorem 4.9]{Suslin}, $K_8(\R)$ is a
uniquely divisible group.
\end{remark}

The following two results are now consequences of Theorems
\ref{th:Kr comp} and \ref{th:n<0} precisely in the same way
as Theorems \ref{th:PTC} and \ref{th:KCst} are consequences
of Theorem \ref{th:Kc comp} and Theorem \ref{th:n<0}.

\begin{theorem}\label{th:PTR}
Let $R$ be any $\Kr$-ring which is $H$-unital as
a $\Q$-algebra. Then the algebraic
$K$-groups are periodic with period 8, the periodicity
isomorphism realized as multiplication by
$v_8\in K_8(\Kr)$.
\end{theorem}

\begin{theorem}\label{th:KRst} \ 
Any Banach $\Kr$-ring $R$ is $K_n$-stable
for every $n\in\Z$. In other words,
the comparison map $\ep_{\ast}\Cl K_{\ast}(R)\to K_{\ast}^\t(R)$
is an isomorphism of $\Z$-graded rings.
\end{theorem}

\bigskip
\section{Hermitian $K$-theory of $\K$-rings}\label{s:KQ}
Let $A$ be a unital ring with anti-involution $\a$.
We shall assume throughout that $(A,\a)$ is \emph{split},
i.e.,\ that there exists an element $\l$ in the center of
$A$ such that $\l+\a(\l)=1$.
This condition is automatically satisfied when
$A$ is a $\Z[\frac12]$-algebra.
One can associate with $(A,\a)$
the Hermitian $K$-theory spectrum $\esKQ(A)$
where $\Ep=\pm1$ (see \cite{KaroubiAnnals} and
\cite{Schlichting} for basic references).
This parallels the construction of the algebraic
$K$-theory spectrum $\sK(A)$ for a unital ring.
(Note that we use
        notation $\sKQ$ instead of $\mathbf L$,
        the notation employed in \cite{KaroubiAnnals} and
	\cite{Berrick-Karoubi},
	in order to avoid confusion with the surgery spectrum
	and the surgery groups.)

The Hermitian $K$-groups $\eKQ_n(A)$ are the homotopy groups
of $\esKQ(A)$ in the same way as Quillen's $K$-groups $K_n(A)$
are the homotopy groups of $\sK(A)$.
For instance, if $A$ is a commutative ring and the anti-involution is
trivial, then
\[
{}_{-1}\KQ_n(A)=\pi_n\LL(\BSp(A)^+\RR)\qquad(n>0)
\]
where $\Sp(A)$ denotes the infinite symplectic group
$\bigcup_{d\geq1}\Sp_{2d}(A)$.

For $n>0$, one has $\eKQ_n(A)=\pi_n\bigl(\BeO(A)^+\bigr)$
where
\[
  \eO(A)=\bigcup_{l>0}\eO_{l,l}(A)
\]
is the group of $\Ep$-orthogonal matrices with
coefficients in $A$ \cite[p.\,64]{KaroubiVillamayor},
\cite[p.\,308]{KaroubiSLN343}.
The classical Whitehead's Theorem
$[\GL(A),\GL(A)]=\operatorname E(A)$ is
replaced by
\[
[\eO(A),\eO(A)]=\eE(A)
\]
where $\eE(A)$ is a suitable replacement
for the group of elementary matrices
\cite[\foreignlanguage{russian}{теоремы 1.4-1.4${}_0$}]{Wasserstein}.

The cone and suspension functors familiar from
algebraic $K$-theory are naturally
equipped with the induced anti-involutions and
the extension
\BE{SCM}
\BD
\node{SA}\node{CA}\arrow{w,,A}\node{MA}\arrow{w,,V}
\ED
\EE
leads to the homotopy fibration
\BE{SCM fibr}
\BD
\node{\BeE(SA)^+}
 \node{\BeO(CA)^+}\arrow{w}
  \node{\BeO(A)^+}\arrow{w}
\ED\,.
\EE
This is a consequence of the fact that
$\eO(CA)$ acts trivially on the homology of
\[
\eO(MA)=\bigcup_{r>0}\eO(M_r(A)).
\]
Indeed, for any unital ring with anti-involution,
the group $\eE(A)$ possesses a system of
generators $E_\b(a)$, labeled by elements
$a\in A$ and by elements $\b$ of a certain
set of ``roots'' $\Phi$ , such that
\[
E_\b(a+a')=E_\b(a)E_\b(a'),\qquad(a,a'\in A),
\]
for every $\b\in\Phi$ \cite[p.\,329]{Wasserstein},
cf.\ also
\cite[3.5(a), p.\,29]{Bak}. In view of this,
for any $c\in CA$ and $r>0$, one can represent $c$ as
$m+c'$ so that $m\in M_r(A)$ and $c'M_r(A)=M_r(A)c'=0$.
In particular, conjugation of elements of $\eO(M_r(A)$
by $E_\b(c)$ coincides with conjugation by the matrix
$E_\b(m)$ which, as an element of $\eO(M_r(A)$,
acts trivially on the homology of
the group to which it belongs.

Acyclicity of $\BeO(CA)^+$, which is proved very
much like the acyclicity of $\BGL(CA)^+$, implies
contractibility of $\BeO(CA)^+$ and yields
functorial isomorphisms
\BE{S iso}
\eKQ_n(A)\simeq\eKQ_{n+1}(SA)\qquad(n>0),
\EE
exactly like in algebraic $K$-theory.
In degrees $n\leq0$, isomorphism \eqref{S iso}
is a consequence of the long exact sequence of
$KQ$-groups, associated with extension \eqref{SCM},
used in conjunction with $\eKQ_n(CA)=0$ and
isomorphisms $\eKQ_n(A)=\eKQ_n(MA)=0$
holding for all $n\in\Z$.

\bigskip
If $A$ is a real or complex Banach algebra
with anti-involution, one can also define
the topological Hermitian $K$-theory spectrum
$\esKQ^\t(A)$ (see \cite{KaroubiSLN343}).
This can be achieved in a manner similar to
how the usual topological $K$-theory spectrum
$\sK^\t(A)$ is built. The definition
easily extends to nonunital Banach algebras.

In \cite{KaroubiSLN343} it was proved that
the Hermitian topological $K$-theory spectrum is 8-periodic:
\[
\esKQ^\t(A)\sim\Om^8(\esKQ^\t(A)).
\]
This homotopy equivalence is, in fact,
induced by the product
with a ``Bott element”
in ${}_1\KQ_8^\t(\R)\simeq\Z\ds\Z$.

\medskip
The functorial \emph{forgetful} and \emph{hyperbolic}
morphisms $\eF$ and $\eH$ connect the $K$-theory and
$\eKQ$-theory spectra
\[
\eF\Cl\esKQ(A) \LA \sK(A),
 \qquad{}\eH\Cl\sK(A) \LA \esKQ(A),
\]
and similarly for topological $K$-theory
\[
\eF^\t\Cl\esKQ^\t(A) \LA \sK^\t(A),
 \qquad\eH^\t\Cl\sK^\t(A) \LA \esKQ^\t(A).
\]
Following \cite{KaroubiAnnals} we denote
the homotopy fibers of $\eF$ and $\eH$
by, respectively, $\esV(A)$ and $\esU(A)$,
and their homotopy groups by, respectively,
\[
\eV_n(A)=\pi_n(\esV(A))
 \AND
\eU_n(A)=\pi_n(\esU(A))\qquad(n\in\Z).
\]
We shall call the cokernel of
the hyperbolic map,
\[
\eW_n(A):=\Coker\LL(K_n(A)
 \xrightarrow{\quad\eH_n\quad}\eKQ_n(A)\RR),
\]
the $n$-th \emph{Witt} group, and the kernel of the
forgetful map,
\[
\eW_n'(A):=\Ker\LL(\eKQ_n(A)
 \xrightarrow{\quad\eF_n\quad}K_n(A)\RR),
\]
the $n$-th \emph{co-Witt} group of a ring
with anti-involution $A$.

In the same way we define,
for a Banach algebra with anti-involution,
the spectra
$\esV^\t(A)$ and $\esU^\t(A)$,
their homotopy groups
$\eV_n^\t(A)$ and $\eU_n^\t(A)$,
and the topological Witt and co-Witt groups.
One can show that
\[
{}_1\mkern-1.4muW_n^\t(A)\simeq K_n^\t(A)
\]
for any real or complex $C^{\ast}$-algebra $A$
with $\a$ being the $*$-operation on $A$
\cite[Definition 2.2 and Théorème 2.3]{Karoubi Annals}.

The following theorem was initially established for
Banach algebras \cite{KaroubiSLN343} and later for general
unital rings with anti-involution \cite{KaroubiAnnals}.
Recently, it was extended to Grothendieck-Witt rings of
pointed pretriangulated DG-categories with weak equivalences
and duality \cite{Schlichting}.

\begin{theorem} \textbf{\textup{(The Fundamental Theorem
of Hermitian $K$-theory)}} \ 
For any unital ring with split anti-involution,
there exists a natural homotopy equivalence of spectra 
\[
\esV(A)\sim\Om(_{-\Ep}\mkern-0.7mu\sU(A)).
\]
 If $A$ is a Banach algebra, then there exists a similar
homotopy equivalence of spectra
\BE{V OmU top}
\esV^\t(A)\sim\Om(_{-\Ep}\mkern-0.7mu\sU^\t(A)),
\EE
and the corresponding functorial comparison morphisms
of spectra making the diagram 
\[
\BD
\node{\esV(A)}\arrow{e}\arrow{s}
 \node{\Om(_{-\Ep}\mkern-0.7mu\sU(A))}\arrow{s}\\
\node{\esV^\t(A)}\arrow{e}
 \node{\Om(_{-\Ep}\mkern-0.7mu\sU^\t(A))}
\ED
\]
commute up to homotopy. In particular,
\[
\eV_n(A)\simeq{}_{-\Ep}U_{n+1}(A)
 \AND
\eV_n^\t(A)\simeq{}_{-\Ep}U_{n+1}^\t(A)\qquad(n\in\Z).
\]
\end{theorem}

\bigskip
We illustrate the Fundamental Theorem in topological
Hermitian $K$-theory by considering Banach algebras
$\R$ and $\C$, equipped with trivial involution, and
the algebra of quaternions $\H$, equipped with
conjugation
\[
w=a+bi+cj+dk\;\LMT\;\bw=a-bi-cj-dk.
\]
The relevant information is collected in Table \ref{tab:RCH}.
The groups in columns 2--4 are filtered unions
of the compact forms of the corresponding Lie groups
and they have the same homotopy type as $\GL(A)$,
\pO\ and \mO. The spaces in the four right columns
represent, up to homotopy, the connected components
of the spaces occurring as degree 0 terms of the
corresponding $\Om$-spectra. The spaces in columns
5 and 7 are, up to connected components, the loop
spaces of their neighbors located in columns 6 and 8,
respectively.

\begin{table}[h]
\caption{} \label{tab:RCH}
 \renewcommand{\arraystretch}{1.4}
\begin{tabular}{@{}ccccccccccccccccc@{}}
\toprule
algebra &&$\GL$&&\pO&&\mO&&\pV&&\mU&&\mV&&\pU\\
\midrule
$\R$ &&   O && \OO  && U  &&  BO && \UO  && \OU  && O  \\
$\C$ &&   U &&  O   && Sp && \UO && \SpU && \USp &&\OU \\
$\H$ &&  Sp && \Spp && U  && BSp && \USp && \SpU && Sp \\[2pt]
\bottomrule
\end{tabular}
\end{table}

If we equip $\C$ with the complex conjugation automorphism
$\s$, or if we consider $A=B\x B^\op$, where $B$ is one
of $\R$, $\C$, or $\H$, with the anti-involution
transposing the factors, then in all four
cases one has \mO(A)$\,=\,$\pO(A).
Accordingly, we use the collective notation
\pmO, \pmV\ and \pmU\ when presenting the
relevant information in Table \ref{tab:RRCCHH}.

The classical homotopy equivalences of the complex
Bott Periodicity correspond to the homotopy equivalences
of \eqref{V OmU top} when $(A,\a)=(\C,\s)$ or
$A=\C\x \C^\op$. Six homotopy equivalences of
Table \ref{tab:RCH} and the remaining two
homotopy equivalences of Table \ref{tab:RRCCHH}
together describe eight homotopy equivalences
of the real Bott Periodicity.

The three equivalences for $A=B\x B^\op$ are
instances of the standard homotopy
equivalence $\text G\sim\Om\text{BG}$, and thus
can be considered ``trivial''. The remaining
seven are not.

\begin{table}[h]
\caption{} \label{tab:RRCCHH}
 \renewcommand{\arraystretch}{1.4}
\begin{tabular}{@{}ccccccccccc@{}}
\toprule
algebra &&$\GL$&&\pmO&&\pmV&&\pmU\\
\midrule
\RRo &&  \OO && O  &&  O && BO  \\
\CCo &&  \UU && U  &&  U && BU  \\
\HHo && \Spp && Sp && Sp && BSp \\
\midrule
$(\C,\s)$&&U&&\UU&&BU&& U  \\
\bottomrule
\end{tabular}
\end{table}

\bigskip
\section{Nonunital algebras with anti-involution}
A standard way to extend an additive functor $G$ from
unital to nonunital $k$-algebras is to consider
\BE{GAk}
G(A)_k:=\Ker\bigl(G(\tA_k)\LA G(k)\bigr)
\EE
where $\tA_k:=k\ltimes A$ is the \emph{unitalization}
of $A$ in the category of $k$-algebras.
The subscript indicates that \eqref{GAk}
in general depends on the choice of $k$-algebra
structure even though the original functor
on the category of unital $k$-algebras may not.
The algebraic $K$-functor is an example.
Additivity of $G$ implies that $G(A)_k$
is canonically isomorphic to $G(A)$ if
$A$ is unital.

\smallskip
Let us say that $(A,\a)$ is a \emph{$(k,\ph)$-algebra
with anti-involution} if $A$ is an algebra over a commutative
unital ring $k$ equipped with an involution $\ph$ such that
anti-involution $\a$ is $\ph$-linear
\[
\a(ca)=\ph(c)\a(a)\qquad(c\in k,\,a\in A).
\]
The corresponding unitalization $\tA_k:=k\ltimes A$ is then
naturally equipped with an anti-involution that extends $\a$
and
\[
\BD
\node{k}\node{\tA_k}\arrow{w,t,A}\pi\node{A}\arrow{w,,V}
\ED
\]
is a split extension in the category of $(k,\ph)$-algebras
with anti-involution. In what follows we will limit ourselves
to the case when $(k,\ph)$ is a split ring with anti-involution.
If $(A,\a)$ admits of such a $(k,\ph)$-algebra structure, we shall say
that $(A,\a)$ is a \emph{split} nonunital ring with anti-involution.

If $G$ is a functor on the category of $(k,\ph)$-algebras
with anti-involution, we shall be denoting the object
defined in \eqref{GAk} by $G(A)_\kf$. If the result
\emph{does not} depend on $\ph$, we shall drop subscript $\ph$.
If it does not depend on $(k,\ph)$, we shall drop
both subscripts and simply write $G(A)$. Thus, we obtain
\[
K_n(A)_k,\quad
 \eKQ_n(A)_\kf,\quad
  \eV_n(A)_\kf,\quad
   \eU_n(A)_\kf,\quad
    \eW_n(A)_\kf,\quad\text{and}\quad\eW_n'(A)_\kf.
\]
Note that any homomorphism $(k,\ph)\to(k',\ph')$ of rings with
anti-involution induces the corresponding homomorphisms
\BE{kk'}
\eKQ_n(A)_\kf \LA \eKQ_n(A)_{k',\ph'}
\EE
and
likewise for $\eV_n$ and $\eU_n$ as well as the Witt
and co-Witt groups. The maps in \eqref{kk'} are instances
of the morphisms to be discussed in a section
devoted to relative Hermitian $K$-theory and excision,
cf.\ \eqref{KQ RA}.

\smallskip
Analogous considerations apply to functors on the category
of Banach algebras with anti-involution. In this case
$k=F$ is the ground field, either $\R$ or $\C$, and
$\ph=\id$ unless $A$ is a $\C$-algebra with a sesquilinear
anti-involution which can happen only when $\ph=\s$,
the complex-conjugation automorphism of $\C$.
The topological groups\label{not depend}
\[
K_n^\t(A)_{F,\ph}\AND\eKQ_n^\t(A)_{F,\ph}\;,
\]
indeed do not depend on $(F,\ph)$ and therefore
will be denoted simply $K_n^\t(A)$ and $\eKQ_n^\t(A)$.

\medskip
A natural transformation of additive functors
\[
\tau\Cl G\LA G'
\]
on the category of $(k,\ph)$-algebras induces
a morphism of functorially split short exact sequences
\BE{GtG'}
\BD
\node{G(k)}\arrow{s,r}{\tau_k}
 \node{G(\tA_k)}\arrow{w,tb,A}{G\pi}{\dashrightarrow}
   \arrow{s,r}{\tau_{\tA_k}}
  \node{G(A)_\kf}\arrow{w,,V}\arrow{s,r}{(\tau_A)_\kf}\\
\node{G'(k)}
 \node{G'(\tA_k)}\arrow{w,tb,A}{G'\pi}{\dashrightarrow}
  \node{G'(A)_\kf}\arrow{w,,V}
\ED
\EE
which implies that $\tau_{\tA_k}$ is isomorphic to
$\tau_k\ds(\tau_A)_\kf$ with respect to the induced
decompositions $G(\tA_k)\simeq G(k)\ds G(A)_\kf$
and $G'(\tA_k)\simeq G'(k)\ds G'(A)_\kf$.
As a consequence, if the target category
has kernels and cokernels, and we set
\[
\G(A):=\Ker\tau_A\AND \G'(A):=\Coker\tau_A\;,
\]
we obtain from \eqref{GtG'} functorially split
short exact sequences, one for $\G$,
\BE{Ga}
\BD
\node{\G(k)}
 \node{\G(\tA_k)}\arrow{w,tb,A}{\G\pi}{\dashrightarrow}
  \node{\Ker(\tau_A)_\kf}\arrow{w,,V}
\ED\;,
\EE
and one for $\G'$,
\BE{Ga'}
\BD
\node{\G'(k)}
 \node{\G'(\tA_k)}\arrow{w,tb,A}{\G'\pi}{\dashrightarrow}
  \node{\Coker(\tau_A)_\kf}\arrow{w,,V}
\ED\;.
\EE
In particular, the functorial morphisms
\BE{Ker Coker}
\Ker(\tau_A)_\kf\xrightarrow{\quad\sim\quad}\G(A)_\kf
 \AND
\Coker(\tau_A)_\kf\xrightarrow{\quad\sim\quad}\G'(A)_\kf
\EE
are isomorphisms.

\smallskip
Another consequence is that $(\tau_A)_\kf$ is an isomorphism
if both $\tau_k$ and $\tau_{\tA_\kf}$ are
isomorphisms.\label{iso->iso}

\smallskip
After making these preliminary remarks, we are ready
to establish the following useful fact.

\begin{prop}\label{pr:W'}
For any Banach algebra $A$, the canonical comparison maps
\BE{W1}
\eW_1(A)_{F,\ph}\LA\eW^\t_1(A)
\EE
and
\BE{W'-1}
\eW'_{-1}(A)_{F,\ph}\LA\eW^{\prime\,\t}_{-1}(A)
\EE
are isomorphisms. \end{prop}

\begin{proof}
The map in \eqref{W1} is an isomorphism
for a unital Banach algebra as was observed
in \cite[pp.\,404--405]{KaroubiSLN343}.
The following commutative diagram
{\small
\BE{12term}
\BD
    \node{\eW'_0(A)}\arrow{s,,=,-}
   \node{k'_0(A)}\arrow{w}\arrow{s,,=,-}
  \node{{}_{-\Ep}W'_{-1}(A)}\arrow{w}\arrow{s}
 \node{\eW_1(A)}\arrow{w}\arrow{s,r}{\simeq}
\node{k_0(A)}\arrow{w}\arrow{s,,=,-}\\
    \node{\eW'_0(A)}
   \node{k'_0(A)}\arrow{w}
  \node{{}_{-\Ep}W^{\prime\,\t}_{-1}(A)}\arrow{w}
 \node{\eW^\t_1(A)}\arrow{w}
\node{k_0(A)}\arrow{w}
\ED
\EE}%
whose rows are portions of a 12-term exact sequence,
cf.\ 
\cite[Théorème 4.3, p.~278]{KaroubiAnnals},
demonstrates the same for the map in \eqref{W'-1}.
Here $k_0(A)$ and $k_0^{^{\prime}}(A)$
denote the even and, respectively, the odd Tate
cohomology groups of $\Ztwo$ acting
on $K_0(A)$ by
\[
[P]\LMT\bigl[P^{\dagger}\bigr]
\]
where $P^{\dagger}$ denotes the module of $\a$-linear
maps $P\to A$
from a finitely generated projective right $A$-module $P$
to $A$.
Recall that the Tate groups of the cyclic group
of order 2 acting on an abelian group $V$
are 2-periodic. The \emph{even} group equals
\[
H^{\text{ev}}(\Ztwo;V)=
 \{v\in V\mid v-\bv=0\}/\{w+\bw\mid w\in V\},
\]
where $w\mapsto\bw$ denotes the action of the generator,
while the \emph{odd} one equals
\[
H^{\text{odd}}(\Ztwo;V)=
 \{v\in V\mid v+\bv=0\}/\{w-\bw\mid w\in V\}.
\]
The nonunital case follows from a remark that preceded
Proposition \ref{pr:W'}.
\end{proof}

\smallskip
As was already mentioned in the context of general
additive functors, if $A$ is a ring with identity,
groups $\eKQ_n(A)_\kf$ are canonically isomorphic
to $\eKQ_n(A)$. In particular, the former do not depend
on the choice of a $(k,\ph)$-algebra structure on $(A,\a)$.
The same holds also for the $\eV_n$ and $\eU_n$ groups,
as well as the Witt and co-Witt groups.

\smallskip
The long exact sequences associated with fibrations
\BE{VKQK UKQK}
\sK \LLA
 \esKQ \LLA
  \esV
\AND
\esKQ \LLA
 \sK \LLA
  \esU
\EE
for unital rings, induce functorial long exact sequences
\BE{VKKQ le}
\cdots \LLA
 \eV_{n-1}(A)_\kf \LLA
  K_n(A)_k \LLA
   \eKQ_n(A)_\kf \LLA
    \eV_n(A)_\kf \LLA
     \cdots,
\EE
and, respectively,
\BE{UKKQ le}
\cdots \LLA
 \eU_{n-1}(A)_\kf \LLA
  \eKQ_n(A)_k \LLA
   K_n(A)_k \LLA
    \eU_n(A)_\kf \LLA
     \cdots.
\EE
While working with the long exact sequences
of \eqref{VKKQ le}--\eqref{UKKQ le},
it is important to bear in mind existence of functorial
isomorphisms
\[
\Coker\LL(K_n(A)_k
 \xrightarrow{\quad\eH_n\quad}\eKQ_n(A)_\kf\RR)
  \;\simeq\; \eW_n(A)_\kf
\]
and
\[
\Ker\LL(\eKQ_n(A)_\kf
 \xrightarrow{\quad\eF_n\quad}K_n(A)_k\RR)
  \;\simeq\; \eW_n'(A)_\kf
\]
which are special instances of the general isomorphisms
of \eqref{Ker Coker}.

\smallskip
There are similar long exact sequences involving
\emph{topological} groups for nonunital Banach algebras.

\section{Induction Theorems}
The two results we are going to
discuss now, collectively called
\emph{Induction Theorems}, reflect a key
feature of Hermitian $K$-theory and admit
a number of variants. Formally speaking,
all those variants sound exactly the
same, the only difference being in the meaning
of the symbols
\BE{KQVUW}
K_n,\qquad\eKQ_n,
 \qquad\eV_n,
  \qquad\eU_n,
   \qquad\eW_n,
    \qquad\eW'_n,
\EE
and
\BE{bKQVUW}
\bK_n,\qquad\ebKQ_n,
 \qquad\ebV_n,
  \qquad\ebU_n,
   \qquad\ebW_n,
    \qquad\ebW'_n,
\EE
present in their formulations.

We begin by considering one situation
served by Induction Theorems: the case of
a homomorphism $f\Cl A\to\bA$ between
unital rings with anti-involution.
In the two theorems that follow,
the lists of symbols in \eqref{KQVUW}--\eqref{bKQVUW}
have the following meaning
\[
K_n=K_n(A),\qquad\eKQ_n=\eKQ_n(A),
 \qquad\eV_n=\eV_n(A),\qquad\dots,
\]
and
\[
\bK_n=K_n(\bA),
 \qquad\ebKQ_n=\ebKQ_n(\bA),
  \qquad\ebV_n=\eV_n(\bA),\qquad\dots,
\]
and the maps between those $K$-groups,
$\eKQ$-groups, $\eV$-groups, etc.,
are all assumed to be induced by $f\Cl A\to\bA$.

\begin{theorem}[Upwards Induction]\label{th:UI}
Let us assume that
\[
K_n\simeq\bK_n,
 \quad K_{n+1}\simeq\bK_{n+1},
  \quad\eKQ_n\simeq\ebKQ_n,
   \aND\eW_{n+1}\simeq\ebW_{n+1}
\]
for both $\Ep=1$ and $-1$.
Then
\BE{UI1}
\eU_n\simeq\ebU_n,
 \qquad\eV_n\simeq\ebV_n,
  \qquad\eU_{n+1}\simeq\ebU_{n+1},
   \qquad\eV_{n-1}\simeq\ebV_{n-1},
\EE
as well as
\BE{UI2}
\eKQ_{n+1}\simeq\ebKQ_{n+1}
 \AND\eW_{n+2}\simeq\ebW_{n+2}
\EE
for both $\Ep=1$ and $-1$.
\end{theorem}

\begin{proof}
We shall analyze a sequence of commutative
diagrams whose rows are portions of the
long exact sequences of homotopy groups of
two fibrations in \eqref{VKQK UKQK} and whose
vertical arrows are the maps induced by $f$.

In view of the hypothesis, the diagram
\[
\BD
    \node{\eKQ_n}\arrow{s,r}{\simeq}
   \node{K_n}\arrow{w}\arrow{s,r}{\simeq}
  \node{\eU_n}\arrow{w}\arrow{s}
 \node{\eKQ_{n+1}}\arrow{w}\arrow{s}
\node{K_{n+1}}\arrow{w}\arrow{s,r}{\simeq}\\
    \node{\ebKQ_n}
   \node{\bK_n}\arrow{w}
  \node{\ebU_n}\arrow{w}
 \node{\ebKQ_{n+1}}\arrow{w}
\node{\bK_{n+1}}\arrow{w}
\ED
\]
yields the diagrams
\[
\BD
    \node{\eKQ_n}\arrow{s,r}{\simeq}
   \node{K_n}\arrow{w}\arrow{s,r}{\simeq}
  \node{\eU_n}\arrow{w}\arrow{s}
 \node{\eW_{n+1}}\arrow{w}\arrow{s,r}{\simeq}
\node{0}\arrow{w}\\
    \node{\ebKQ_n}
   \node{\bK_n}\arrow{w}
  \node{\ebU_n}\arrow{w}
 \node{\ebW_{n+1}}\arrow{w}
\node{0}\arrow{w}
\ED
\]
and
\[
\BD
   \node{0}
  \node{\eW_{n+1}}\arrow{w}\arrow{s,r}{\simeq}
 \node{\eKQ_{n+1}}\arrow{w}\arrow{s}
\node{K_{n+1}}\arrow{w}\arrow{s}\\
   \node{0}
  \node{\ebW_{n+1}}\arrow{w}
 \node{\ebKQ_{n+1}}\arrow{w}
\node{\bK_{n+1}}\arrow{w}
\ED\;.
\]
The Five Lemma implies that $\eU_n\simeq\ebU_n$ and
$\eKQ_{n+1}\onto\ebKQ_{n+1}$. By applying it again to the diagram
\[
\BD
    \node{K_n}\arrow{s,r}{\simeq}
   \node{\eKQ_n}\arrow{w}\arrow{s,r}{\simeq}
  \node{\eV_n}\arrow{w}\arrow{s}
 \node{K_{n+1}}\arrow{w}\arrow{s,r}{\simeq}
\node{\eKQ_{n+1}}\arrow{w}\arrow{s,,A}\\
    \node{\bK_n}
   \node{\ebKQ_n}\arrow{w}
  \node{\ebV_n}\arrow{w}
 \node{\bK_{n+1}}\arrow{w}
\node{\ebKQ_{n+1}}\arrow{w}
\ED\;.
\]
we obtain $\eV_n\simeq\ebV_n$.

The remaining two isomorphisms in \eqref{UI1} follow from
the two proven ones in view of the Fundamental Theorem
of Hermitian $K$-theory.

The Five Lemma applied to the diagrams
\[
\BD
    \node{K_n}\arrow{s,r}{\simeq}
   \node{\eU_n}\arrow{w}\arrow{s,r}{\simeq}
  \node{\eKQ_{n+1}}\arrow{w}\arrow{s}
 \node{K_{n+1}}\arrow{w}\arrow{s,r}{\simeq}
\node{\eU_{n+1}}\arrow{w}\arrow{s,r}{\simeq}\\
    \node{\bK_n}
   \node{\ebU_n}\arrow{w}
  \node{\ebKQ_{n+1}}\arrow{w}
 \node{\bK_{n+1}}\arrow{w}
\node{\ebU_{n+1}}\arrow{w}
\ED
\]
and
\[
\BD
    \node{\eKQ_{n+1}}\arrow{s,r}{\simeq}
   \node{K_{n+1}}\arrow{w}\arrow{s,r}{\simeq}
  \node{\eU_{n+1}}\arrow{w}\arrow{s,r}{\simeq}
 \node{\eW_{n+2}}\arrow{w}\arrow{s}
\node{0}\arrow{w}\\
    \node{\ebKQ_{n+1}}
   \node{\bK_{n+1}}\arrow{w}
  \node{\ebU_{n+1}}\arrow{w}
 \node{\ebW_{n+2}}\arrow{w}
\node{0}\arrow{w}
\ED
\]
yields the two isomorphisms of \eqref{UI2}.
\end{proof}

\begin{cor}\label{cor:UI}
If
\[
\eKQ_{l}\simeq\ebKQ_{l}\AND
 \eKQ_{l+1}\simeq\ebKQ_{l+1},
\]
for a given $l$ and for both $\Ep=1$ and $-1$,
and if
\[
  K_n\simeq\bK_n\qquad(n\geq l),
\]
then
\[
\eKQ_n\simeq\ebKQ_n,\qquad
 \eV_n\simeq\ebV_n,\AND
  \eU_n\simeq\ebU_n,
\]
for $n\geq l$ and both $\Ep=1$ and $-1$.
\end{cor}\qed

\begin{theorem}[Downwards Induction]\label{th:DI}
Let us assume that
\[
K_n\simeq\bK_n,
 \quad K_{n-1}\simeq\bK_{n-1},
  \quad\eKQ_n\simeq\ebKQ_n,
   \aND\eW'_{n-1}\simeq\ebW'_{n-1}
\]
for both $\Ep=1$ and $-1$.
Then
\BE{DI1}
\eV_{n-1}\simeq\ebV_{n-1},
 \qquad\eU_{n-1}\simeq\ebU_{n-1},
  \qquad\eU_n\simeq\ebU_n,
   \qquad\eV_{n-2}\simeq\ebV_{n-2},
\EE
as well as
\BE{DI2}
\eKQ_{n-1}\simeq\ebKQ_{n-1}
 \AND\eW'_{n-2}\simeq\ebW'_{n-2}
\EE
for both $\Ep=1$ and $-1$.
\end{theorem}

\begin{proof} The proof is analogous. We provide the details
for the reader's convenience. In view of the hypothesis, the diagram
\[
\BD
    \node{K_{n-1}}\arrow{s,r}{\simeq}
   \node{\eKQ_{n-1}}\arrow{w}\arrow{s}
  \node{\eV_{n-1}}\arrow{w}\arrow{s}
 \node{K_n}\arrow{w}\arrow{s,r}{\simeq}
\node{\eKQ_n}\arrow{w}\arrow{s,r}{\simeq}\\
    \node{\bK_{n-1}}
   \node{\ebKQ_{n-1}}\arrow{w}
  \node{\ebV_{n-1}}\arrow{w}
 \node{\bK_n}\arrow{w}
\node{\ebKQ_n}\arrow{w}
\ED
\]
yields the diagrams
\[
\BD
    \node{0}
   \node{\eW'_{n-1}}\arrow{w}\arrow{s,r}{\simeq}
  \node{\eV_{n-1}}\arrow{w}\arrow{s}
 \node{K_n}\arrow{w}\arrow{s,r}{\simeq}
\node{\eKQ_n}\arrow{w}\arrow{s,r}{\simeq}\\
    \node{0}
   \node{\ebW'_{n-1}}\arrow{w}
  \node{\ebV_{n-1}}\arrow{w}
 \node{\bK_n}\arrow{w}
\node{\ebKQ_n}\arrow{w}
\ED
\]
and
\[
\BD
   \node{K_{n-1}}\arrow{s,r}{\simeq}
  \node{\eKQ_{n-1}}\arrow{w}\arrow{s}
 \node{\eW'_{n-1}}\arrow{w}\arrow{s,r}{\simeq}
\node{0}\arrow{w}\\
   \node{\bK_{n-1}}
  \node{\ebKQ_{n-1}}\arrow{w}
 \node{\ebW'_{n-1}}\arrow{w}
\node{0}\arrow{w}
\ED\;.
\]
The Five Lemma implies that $\eV_{n-1}\simeq\ebV_{n-1}$ and
$\eKQ_{n+1}\into\ebKQ_{n-1}$. By applying it again to the diagram
\[
\BD
    \node{\eKQ_{n-1}}\arrow{s,,V}
   \node{K_{n-1}}\arrow{w}\arrow{s,r}{\simeq}
  \node{\eU_{n-1}}\arrow{w}\arrow{s}
 \node{\eKQ_n}\arrow{w}\arrow{s,r}{\simeq}
\node{K_n}\arrow{w}\arrow{s,r}{\simeq}\\
    \node{\ebKQ_{n-1}}
   \node{\bK_{n-1}}\arrow{w}
  \node{\ebU_{n-1}}\arrow{w}
 \node{\ebKQ_n}\arrow{w}
\node{\bK_n}\arrow{w}
\ED
\]
we obtain $\eU_{n-1}\simeq\ebU_{n-1}$.

The remaining two isomorphisms in \eqref{DI1} follow from
the two proven ones in view of the Fundamental Theorem
of Hermitian $K$-theory.

The Five Lemma applied to the diagrams
\[
\BD
    \node{K_n}\arrow{s,r}{\simeq}
   \node{\eU_n}\arrow{w}\arrow{s,r}{\simeq}
  \node{\eKQ_{n-1}}\arrow{w}\arrow{s}
 \node{K_{n-1}}\arrow{w}\arrow{s,r}{\simeq}
\node{\eU_{n-1}}\arrow{w}\arrow{s,r}{\simeq}\\
    \node{\bK_n}
   \node{\ebU_n}\arrow{w}
  \node{\ebKQ_{n-1}}\arrow{w}
 \node{\bK_{n-1}}\arrow{w}
\node{\ebU_{n-1}}\arrow{w}
\ED
\]
and
\[
\BD
    \node{0}
   \node{\eW'_{n-2}}\arrow{w}\arrow{s}
  \node{\eV_{n-2}}\arrow{w}\arrow{s,r}{\simeq}
 \node{K_{n-1}}\arrow{w}\arrow{s,r}{\simeq}
\node{\eKQ_{n-1}}\arrow{w}\arrow{s,r}{\simeq}\\
    \node{0}
   \node{\ebW'_{n-2}}\arrow{w}
  \node{\ebV_{n-2}}\arrow{w}
 \node{\bK_{n-1}}\arrow{w}
\node{\ebKQ_{n-1}}\arrow{w}
\ED
\]
yields the two isomorphisms of \eqref{DI2}.
\end{proof}

\begin{cor}\label{cor:DI}
If
\[
\eKQ_{l}\simeq\ebKQ_{l}\AND
 \eKQ_{l-1}\simeq\ebKQ_{l-1},
\]
for a given $l$ and for both $\Ep=1$ and $-1$, and if
\[
  K_n\simeq\bK_n\qquad(n\leq l),
\]
then
\[
\eKQ_n\simeq\ebKQ_n,\qquad
 \eV_n\simeq\ebV_n,\AND
  \eU_n\simeq\ebU_n\qquad
\]
for $n\leq l$ and both $\Ep=1$ and $-1$.
\end{cor}\qed

\subsection*{Nonunital and topological variants}
Suppose $(A,\a)$ is a nonunital $(k,\ph)$-algebra
with anti-involution and $(\bA,\FineBar{1.5mu}{0.5mu}\a)$ is
a nonunital $(\bar k,\bar\ph)$-algebra with
anti-involution, and $(f,f_k)$ is a corresponding
homomorphism, i.e.,\ a pair of homomorphisms
of rings with anti-involution $f\Cl A\to\bA$
and $f_k\Cl k\to\bar k$ such that
\[
f(ca)=f_k(c)f(a)\qquad(c\in k;\;a\in A).
\]
Theorems \ref{th:UI} and \ref{th:DI} hold
if
\[
K_n=K_n(A)_\kf,\qquad\eKQ_n=\eKQ_n(A)_\kf,
 \qquad\eV_n=\eV_n(A)_\kf,\qquad\dots,
\]
and
\[
\bK_n=K_n(\bA)_\kf,
 \qquad\ebKQ_n=\ebKQ_n(\bA)_\kf,
  \qquad\ebV_n=\eV_n(\bA)_\kf,\qquad\dots,
\]
and the maps between $K$-groups, $\eKQ$-groups,
$\eV$-groups, etc., are all assumed to be induced by $(f,f_k)$.
We shall refer to these theorems as
the \emph{Nonunital Homomorphism Induction Theorems}
in order to distinguish them from the original
formulations of Theorems \ref{th:UI} and
\ref{th:DI} which we shall
refer to as
the \emph{Unital Homomorphism Induction Theorems}.
They occur as special cases of more general theorems
involving \emph{relative} $K$-groups, $\eKQ$-groups, $\eV$-groups,
etc.
Properly introducing the corresponding relative versions
of all the objects would significantly increase
the size of this chapter while it is not needed for
the results in the present article.

\smallskip
Then there are obvious versions of both pairs of induction
theorems for Banach algebras and topological $K$-theory.
We shall treat them as a single pair of theorems
acknowledging the fact that $K_{\ast}^\t(\ )_F$ and
$\eKQ_{\ast}^\t(\ )_{F,\ph}$\,, and the functors derived from
them, do not depend on the $(F,\ph)$-algebra structure,
as was already pointed out on p.\,\pageref{not depend}.
We shall refer to this pair as the \emph{Topological}
(or, \emph{Banach}) \emph{Induction Theorems}.

This produces six theorems (plus six corollaries) in total.
Now, one can replace the ordinary homotopy groups by
the homotopy groups with coefficients, e.g.,
finite or rational. This will produce the
corresponding versions ``with coefficients''
of all those theorems and their corollaries.
Several such results have been previously employed,
cf.\ e.g.,\ \cite{Berrick-Karoubi} and \cite{Battikh}.
In the next section we will encounter yet another variety:
the \emph{Comparison Induction Theorems} for Banach algebras.

\bigskip
\section{The comparison map in Hermitian $K$-theory}
There are functorial comparison maps
for nonunital Banach algebras,
\BE{comp KQ}
\eKQ_{\ast}(A)_\kf\LA\eKQ_{\ast}^\t(A),
\EE
and similar comparison maps for
$\eV_{\!\ast}$, $\eU_{\ast}$, $\eW_{\ast}$ and $\eW_{\ast}'$,
where $k\sbeq F$ denotes any subring of
the ground field and $\ph=\id$ unless $F=\C$
and the anti-involution on $A$ is sesquilinear.
In the last case $k$ is supposed to be a subring of
$\C$ invariant under complex conjugation
and $\ph$ is complex conjugation restricted to $k$.

We shall study those maps in relation to the
comparison maps in algebraic $K$-theory,
\BE{comp K}
K_n(A)_k\LA K_n^\t(A),
\EE
using yet another variety of Induction Theorems.
In the \emph{Comparison Induction Theorems} only
a single Banach algebra $A$ is present,
and the groups in \eqref{bKQVUW} are the topological
counterparts of the groups in \eqref{KQVUW}.
The comparison maps of \eqref{comp K} play
the role of the maps $K_n\to\bK_n$,
the comparison maps of \eqref{comp KQ} play
the role of the maps $\eKQ_n\to\ebKQ_n$,
etc.

The following theorem is a direct consequence
of Proposition \ref{pr:W'} combined with
the Comparison Induction Theorems.

\begin{theorem}\label{K KQ comp}
Let $A$ be a Banach algebra with anti-involution.
If the comparison maps of \eqref{comp K} are
isomorphisms in the range $0<n\leq n'$ \textup(respectively, in
the range $n''\leq n<0$\textup), then the comparison maps
of \eqref{comp KQ} are isomorphisms in precisely
the same range.
\end{theorem}

Equipped with Theorem \ref{K KQ comp} we deduce the following
result from Theorem \ref{th:n<0}.

\begin{theorem}\label{th:n<0 KQ}
If $A$ is a stable retract of
a Banach $\Kr$ or $\Kc$-ring,
then, for any anti-involution $\a$ on $A$, the comparison
maps in Hermitian $K$-theory, \eqref{comp KQ},
are isomorphisms for $n\leq0$.\qed
\end{theorem}

The next result is similarly deduced from Theorems
\ref{th:KCst} and \ref{th:KRst}

\begin{theorem}\label{th:st KQ}
If $A$ is a stable retract of
a Banach $\Kr$ or $\Kc$-ring that is $H$-unital over $\Q$,
then, for any anti-involution $\a$ on $A$, the comparison
maps in Hermitian $K$-theory, \eqref{comp KQ},
are isomorphisms for $n\in\Z$.\qed
\end{theorem}

\smallskip
The case of a complex stable $C^{\ast}$-algebra was
deduced from the results of \cite{Suslin-Wodzicki}
by Battikh \cite{Battikh}.

\bigskip
All rings satisfy excision in $K$-theory in degrees
less than or equal 0, while rings satisfying
the hypothesis of Theorem \ref{th:st KQ}
satisfy excision in all degrees. In the next section we prove
that such rings satisfy excision also in Hermitian $K$-theory
which explains why the groups $\eKQ_{\ast}(A)_\kf$ are,
according to Theorem \ref{th:st KQ}, all isomorphic
to each other. 

\bigskip
\section{Relative $\eKQ$-groups and Excision}

\smallskip
If $A\sbeq R$ is a two-sided ideal in a unital
ring with anti-involution $\a$ and $A$ is $\a$-invariant,
then the relative $\eKQ$-groups are defined in terms of
the homotopy fiber $\mathbf F(R,A)$
of the morphism of $\Om$-spectra,
\[
\esKQ(R) \LA \esKQ(R/A)),
\]
induced by the quotient homomorphism
$R\to R/A$.
The relative $\eKQ$-groups,
\BE{KRA}
\eKQ_n(R,A):=\pi_n(\mathbf F(R,A))\qquad(n\in\Z),
\EE
become functors on the category of triples $(R,A,\a)$.
One has
\[
\eKQ_n(R,A)=\pi_n(F(R,A))\qquad(n>0)
\]
where $F(R,A)$ denotes the homotopy fiber of the map
\[
\BeO(R)^+\LA\BeO(R/A)^+.
\]
The connected component of $F(R,A)$ is naturally
identified with $\bF(R,A)$, the homotopy fiber
of
\[
\BeO(R)^+\LA\BebO(R/A)^+,
\]
where $\ebO(R/A)$ denotes the image of $\eO(R)$ in
$\eO(R/A)$ and
\[
\pi_0(F(R,A))=\Coker\bigl(\eKQ_1(R)\LA\eKQ_1(R/A)\bigr)
 \simeq\Ker\bigl(\eKQ_0(R,A)\LA\eKQ_0(R)\bigr),
\]
while the homotopy equivalence
\[
F(R,A)\;\sim\;\Om\FE(SR,SA),
\]
where $\FE(R,A)$ is simultaneously
the homotopy fiber of the map
\[
\BeE(R)^+\LA\BeE(R/A)^+,
\]
and a covering of $\bF(R,A)$,
induces suspension isomorphisms
\[
\pi_n(R,A)\;\simeq\;\pi_{n+1}(SR,SA)\qquad(n>0).
\]

\bigskip
For any homomorphism of unital rings with anti-involution
$f\Cl R\to R'$, let us consider, following \cite{Wagoner},
the ring $\G(f)$ defined by the pull-back diagram
\BE{G sq}
\BD
\node{CR'}\arrow{s,,A}
 \node{\G(f)}\arrow{w,,..}\arrow{s,,..}\\
\node{SR'}
 \node{SR}\arrow{w,,A}
\ED\,.
\EE
The diagram in \eqref{G sq} induces the following
commutative diagram of group
extensions
\[
\BD
\node{\eO(MR')}\arrow{s,,V}
 \node{\eO(MR')}\arrow{w,,=}\arrow{s,,V}\\
\node{\eO(CR')}\arrow{s,A}
 \node{\eO(\G(f))}\arrow{w,}\arrow{s,,A}\\
\node{\eE(SR')}
 \node{\ebO(SR)}\arrow{w,}
\ED\]
where $\ebO(SR)$ is the image of $\eO(\G(f))$ in $\eO(SR)$.
The group $\eO(\G(f))$ acts trivially on the homology of
$\eO(MR')$ because the action factorizes through the
action of $\eO(CR')$ which is trivial, as we pointed
above. It follows that
\[
\BD
\node{\BeE(SR')^+}
 \node{\BebO(SR)^+}\arrow{w}
  \node{\BeO(\G(f))^+}\arrow{w}
   \node{\BeO(R')^+}\arrow{w}
\ED
\]
is an initial segment of the Puppe sequence of
homotopy fibrations.

\medskip
In the special case when $f\Cl R\to R/A$ is the
quotient homomorphism, we obtain the natural map
\BE{FRA->GRA}
\BD
\node{F(R,A)}\arrow{e,t,}\sim
 \node{\Om\FE(SR,SA)}\arrow{e}
  \node{\Om\BebO(\G(R,A))^+}
\ED
\EE
where $\G(R,A):=\G(f)$.  It induces
isomorphism on $\pi_n$, for $n>0$, and is
injective on $\pi_0$. In particular,
a morphism $\phi\Cl(R,A)\to(R',A')$
induces an isomorphism of relative
$\eKQ$-groups for a particular $n>0$,
\BE{KQ RA}
\eKQ_n(R,A)\;\simeq\;\eKQ_n(R',A'),
\EE
if and only if
\BE{KQ GRA}
\eKQ_{n+1}(\G(R,A))\;\simeq\;\eKQ_{n+1}(\G(R',A')).
\EE

All the arguments in this section have
exact counterparts for the general linear
group $\GL$, the group of elementary
matrices $\E$, and algebraic $K$-theory.
In particular,
\BE{K RA}
K_n(R,A)\;\simeq\;K_n(R',A'),
\EE
for a given $n>0$, if and only if
\BE{K GRA}
K_{n+1}(\G(R,A))\;\simeq\;K_{n+1}(\G(R',A')).
\EE

\medskip
Recall that a nonunital ring $A$ is said
to satisfy excision for $K_n$ if
any morphism
\BE{exc}
\phi\Cl(R,A)\LA(R',A)
\EE
which on $A$ restricts to the identity map,
induces an isomorphism on $K_n$
\BE{Kn exc}
\BD
\node{K_n(R,A)}\arrow{e,tb}{K_n\phi}{\sim}\node{K_n(R',A)}
\ED\,.
\EE
In the same way, one can define excision for
any other functor on the category of pairs $(R,A)$.
Analogously, one defines excision for nonunital
rings with anti-involution and $\eKQ$.

According to Bass' Excision Theorem
\cite[Sections VII.6 and XII.8]{Bassbook}, every
ring satisfies excision for $K_n$ and $n\leq0$. A similar
result holds for nonunital rings with anti-involution
\cite[Section 3, Théorème 4.1]{KaroubiVillamayor}.
The following result demonstrates usefulness of
the Induction Theorems.

\begin{theorem}[cf.\ \cite{BattikhCR}, \cite{Battikh}]
If a ring $A$ satisfies excision for $K_n$, $n\leq n'$,
then, for any anti-involution on $A$ which admits
a split unitalization, $(A,\a)$ satisfies excision
in $\eKQ_n$ in the same range $n\leq n'$.
\end{theorem}

\begin{proof}
Consider the commutative diagram of ring extensions
\[
\BD
\node{C(R/A)}\arrow{s}
 \node{\G(R,A)}\arrow{w,,A}\arrow{s,r}{\G\phi}
  \node{SA}\arrow{w,,V}\arrow{s,,=}\\
\node{C(R'/A)}
 \node{\G(R',A)}\arrow{w,,A}
  \node{SA}\arrow{w,,V}
\ED
\]
where $\G\phi$ is induced by \eqref{exc}.
The associated long exact sequences of $K$-groups
and $\eKQ$-groups produce two sequences of commutative
squares, $n\in\Z$,
\BE{Kn sq}
\BD
 \node{K_n(\G(R,A))}\arrow{s,r}{\G\phi}
  \node{K_n(\G(R,A),SA)}\arrow{w,t}\sim\arrow{s}\\
 \node{K_n(\G(R',A))}
  \node{K_n(\G(R',A),SA)}\arrow{w,t}\sim
\ED
\EE
and
\BE{KQn sq}
\BD
 \node{\eKQ_n(\G(R,A))}\arrow{s,r}{\G\phi}
  \node{\eKQ_n(\G(R,A),SA)}\arrow{w,t}\sim\arrow{s}\\
 \node{\eKQ_n(\G(R',A))}
  \node{\eKQ_n(\G(R',A),SA)}\arrow{w,t}\sim
\ED\,.
\EE
In view of Bass' Excision Theorem, right vertical arrows
in \eqref{Kn sq} are isomorphisms for $n\leq0$. Similarly
for \eqref{KQn sq}, in view of \cite[Section 3,
Théorème 4.1]{KaroubiVillamayor}.

Equivalence of \eqref{K RA} and \eqref{K GRA} used
in conjunction
with the Upwards Induction Theorem proves that $\G\phi$ induces
isomorphisms $\eKQ_{n+1}(\G(R,A))\simeq\eKQ_{n+1}(\G(R',A))$
for $n\leq n'$.  Invoking equivalence of \eqref{KQ GRA} and
\eqref{KQ RA} completes the proof.
\end{proof}

\appendix

\bigskip
\section{Multiplicative structures in $K$-theory}\label{Ap MSK}

\smallskip
\subsection{Product in $K$-theory}
We say that a pairing between rings 
\BE{bm-p}
\ph:A\x B \LA C
\EE
is \emph{bimultiplicative} if
\[
\ph(aa',bb')=\ph(a,b)\ph(a',b')
 \qquad(a,a'\in A;\,b,b'\in B).
\]
If rings are unital, we say that the pairing is
\emph{unital} if
\[
\ph(1_A,1_B)=1_C.
\]
A unital bimultiplicative pairing induces pairings
between algebraic $K$-groups 
\BE{K-p}
K_m(A)\x K_n(B) \LA K_{m+n}(C)
 \qquad(m,n\in\Z).
\EE
A detailed treatment can be found, for instance, in
\cite[pp.\,219-227]{Karoubi Annals}.
The pairings in \eqref{K-p} are associative,
graded-commutative and functorial.
One calls them collectively
the \emph{product structure} in algebraic
$K$-theory of unital rings,

\smallskip
Similarly, a \emph{continuous}
unital bimultiplicative pairing between
Banach algebras induces pairings
between topological $K$-groups
\begin{equation}\label{Kt-p}
K_m^\t(A)\x K_n^\t(B) \LA K_{m+n}^\t(C)
 \qquad(m,n\in\Z),
\end{equation}
and the comparison map between algebraic
and topological $K$-groups carries
the pairings in \eqref{K-p} to
the pairings in \eqref{Kt-p},
cf.\ \cite[Sections 1.25 and 2.24]{Karoubi Annals}.

\medskip
The algebraic $K$-groups of a non unital ring $A$
are usually defined as
\[
K_n(A):=\Ker\LL(K_n(\tA) \LA K_n(\Z)\RR)
\]
where $\tA:=\Z\ltimes A$ denotes the \emph{unitalization}
of ring $A$.
If $A$ is a $k$-algebra over a unital ring $k$,
then one can use instead the unitalization
$\tA_k:=k\ltimes A$ in the category of
$k$-algebras:
\[
K_n(A)_{k}:=\Ker\LL(K_n(\tA_k) \LA K_n(k)\RR).
\]
The corresponding $K$-groups, however,
depend on the choice of $k$ if $n>0$.

\smallskip
A bimultiplicative pairing \eqref{bm-p} between nonunital rings
does not extend to a pairing 
\[
\tA_\Z\x\tB_\Z \LA \tC_\Z.
\]
This is underscored by the fact that the universal pairing
\[
A\x B\LA A\T_\Z B,\qquad(a,b)\LMT a\T b\qquad(a\in A\;\,b\in B),
\]
induces the unital pairing
\[
\tA_\Z\x\tB_\Z\LA\tA_\Z\T_\Z\tB_\Z
\]
rather than
\[
\tA_\Z\x\tB_\Z\LA(A\T_\Z B)^{\kern-0.6pt\Til}_\Z
\] 
and it, accordingly, induces pairings between
the relative $K$-groups
\[
K_m(A)\x K_n(B) \LA
 K_{m+n}\bigl(\tA_\Z\T_\Z\tB_\Z,A\T_\Z B\bigr)
\]
instead of
\BE{univ K p}
K_m(A)\x K_n(B) \LA K_{m+n}\bigl(A\T_\Z B\bigr).
\EE
The canonical map
\BE{exc map}
K_{\ast}\bigl(A\T_\Z B\bigr)
 \LA K_{\ast}\bigl(\tA_\Z\T\tB_k,A\T_\Z B\bigr)
\EE
is an isomorphism in degrees less than or equal 0, and
rarely so in positive degrees.
If the ring $A\T_\Z B$ satisfies excision in $K$-theory, then
the map in \eqref{exc map} is an isomorphism in all degrees, and
any bimultiplicative pairing \eqref{bm-p} induces binary
pairings
\[
K_m(A)\x K_n(B) \LA K_{m+n}(C).
\]
Questions of associativity for such pairings would,
however, depend on whether the triple tensor
products $A\T_\Z B\T_\Z C$ satisfy excision as well.

Rings satisfying excision in rational algebraic $K$-theory
were completely characterized in a pair of articles
\cite{Wodzicki.Annals} and \cite{Suslin-Wodzicki}.
It was also proved that a $\Q$-algebra satisfies
excision in $K$-theory if and only if it
is $H$-unital. The category of $H$-unital $\Q$-algebras
is closed under $\T_\Z$,
cf.\ \cite[Theorem 7.10]{Suslin-Wodzicki}.
This in turn implies that the universal pairings of
\eqref{univ K p} are for such rings associative.

\bigskip
\section{Bounded approximate identities and $\pT$}\label{Ap BAI}
We collect here some important facts that relate
presence of a bounded approximate identity in a Banach
algebra to certain exactness properties of the
projective tensor product of Banach spaces,
introduced by Schatten \cite{Schatten}.
These results are quite well known, and apply
equally to complex and real algebras. In particular,
the terms \emph{normed algebra}, \emph{Banach algebra}
and \emph{Banach space} will be used below collectively,
to cover the complex and the real cases alike.

We shall say that an extension in the category
of Banach spaces (with bounded linear maps
	as morphisms),
\BE{DEF}
\begin{diagram}
 \node{D}\node{E}\arrow{w,t,A}{p}\node{F}\arrow{w,t,V}{i}
\end{diagram},
\EE
is \emph{pure} or, more precisely, $\pT$-{pure},
if the functor $C\pT$ preserves exactness of \eqref{DEF}
for any Banach space $C$.

\begin{lemma}
The following conditions are equivalent:
\begin{enumerate}
 \renewcommand\theenumi{\alph{enumi}}
\item the extension in \eqref{DEF} is pure;
\item the sequence
\BE{F*DEF}
\BD\dgARROWLENGTH=2.0em
\node0
 \node{F^{\ast}\pT D}\arrow{w}
  \node[2]{F^{\ast}\pT E}\arrow[2]{w,t}{\id_{F^{\ast}}\pT p}
   \node[2]{F^{\ast}\pT F}\arrow[2]{w,t}{\id_{F^{\ast}}\pT i}
    \node0\arrow{w}
\ED
\EE
  is exact;
\item the dual extension
\BE{DEF*}
\BD
 \node{D^{\ast}}\arrow{e,t,V}{p^{\ast}}\node{E^{\ast}}\arrow{e,t,A}{i^{\ast}}\node{F^{\ast}}
\ED
\EE
is split.\footnote{In the existing literature on Banach
	spaces such extensions are often called \emph{weakly split}.}
\end{enumerate}
\end{lemma}

\begin{proof}
Since a one-dimensional Banach space is an injective cogenerator
in the category of Banach spaces (the Hahn-Banach theorem),
the sequence
\BE{C DEF}
\BD\dgARROWLENGTH=2.0em
\node0
 \node{C\pT D}\arrow{w}
  \node[2]{C\pT E}\arrow[2]{w,t}{\id_C\pT p}
   \node[2]{C\pT F}\arrow[2]{w,t}{\id_C\pT i}
    \node0\arrow{w}
\ED
\EE
is exact if and only if
\BE{C DEF*}
\BD\dgARROWLENGTH=2.0em
\node0\arrow{e}
 \node{(C\pT D)^{\ast}}\arrow[2]{e,t}{(\id_C\pT p)^{\ast}}
  \node[2]{(C\pT E)^{\ast}}\arrow[2]{e,t}{(\id_C\pT i)^{\ast}}
   \node[2]{(C\pT F)^{\ast}}\arrow{e}
    \node0
\ED
\EE
is exact. Since $(X\pT Y)^{\ast}$ is isometric to
the space $\L(X,Y^{\ast})$ of bounded linear operators
from $X$ to $Y^{\ast}$, the sequence in \eqref{C DEF*}
coincides with
\BE{C->DEF*}
\BD\dgARROWLENGTH=2.0em
\node0\arrow{e}
 \node{\L(C,D^{\ast})}\arrow[2]{e,t}{p^{\ast}\c}
  \node[2]{\L(C,E^{\ast})}\arrow[2]{e,t}{i^{\ast}\c}
   \node[2]{\L(C,F^{\ast})}\arrow{e}
    \node0
\ED\;.
\EE
If \eqref{DEF*} is split, then
\eqref{C->DEF*}, and thus also \eqref{C DEF*},
are split-exact.
In particular, \eqref{C DEF} is exact.
In reverse, if \eqref{F*DEF} is exact, then
its dual
\BE{F->DEF*}
\BD\dgARROWLENGTH=2.0em
\node0\arrow{e}
 \node{\L(F^{\ast},D^{\ast})}\arrow[2]{e,t}{p^{\ast}\c}
  \node[2]{\L(F^{\ast},E^{\ast})}\arrow[2]{e,t}{i^{\ast}\c}
   \node[2]{\L(F^{\ast},F^{\ast})}\arrow{e}
    \node0
\ED
\EE
is exact which implies that \eqref{DEF*} is split.
\end{proof}

An example of a pure-exact extension is provided
by
\BE{F** F}
\begin{diagram}
 \node{F^{\astwo}/F}\node{F^{\astwo}}\arrow{w,t,A}{\pi}\node{F}\arrow{w,t,V}{\k}
\end{diagram}
\EE
where $\k$ denotes the canonical isometric embedding of an
arbitrary Banach space $F$ into its second dual. Pure-exactness 
of \eqref{F** F} follows from the fact that
\[
\k^{\ast}\Cl X^{\asthree}\LA X^{\ast}
\]
has a canonical splitting of norm 1.
The extension in \eqref{F** F} is split precisely
when $F$ is isomorphic to a complemented subspace of $E^{\ast}$
for some Banach space $E$.

\medskip
Recall that a net $(e_i)_{i\in I}$ in a normed algebra $A$ is
a \emph{left approximate identity} if
\BE{LAI}
\lim_{i\in I}\|e_ia-a\|=0\qquad\text{for any $a\in A$.}
\EE
A right approximate identity is defined similarly.
We say that the approximate identity is \emph{bounded}
if $\sup_{i\in I}\|e_i\|<\infty$.

Among numerous examples of normed algebras with
bounded approximate identity one should mention
that every right (respectively, left) ideal
in a $C^{\ast}$-algebra possesses a bounded left
(respectively, right) approximate identity
\cite[1.7.3]{Dixmier}.

\begin{prop}\label{pr:BAI}
For any Banach algebra $A$, the following conditions
are equivalent:
\begin{enumerate}
 \renewcommand\theenumi{\alph{enumi}}
\item for any unital Banach algebra $B$
 which contains $A$ as a closed right ideal,
 the extension 
\BE{B A*}
\begin{diagram}
 \node{(B/A)^{\ast}}\arrow{e,,V}\node{B^{\ast}}\arrow{e,,A}\node{A^{\ast}}
\end{diagram}
\EE
admits a bounded $B$-linear splitting;
\item there exists a unital Banach algebra $B_0$
 which contains $A$ as a closed right ideal,
 such that the extension 
\BE{B0 A**}
\begin{diagram}
 \node{(B_0/A)^{\astwo}}\node{B_0^{\astwo}}\arrow{w,,A}\node{A^{\astwo}}\arrow{w,,V}
\end{diagram}
\EE
admits a bounded $B_0$-linear splitting;
\item $A$ possesses a bounded left approximate identity.
\end{enumerate}
\end{prop}

\begin{proof}
Let us consider the commutative diagram of extensions
of Banach right $B_0$-modules
\BE{BA**}
\begin{diagram}
 \node{B_0/A}\arrow{s,r,V}{\k}
  \node{B_0}\arrow{s,r,V}{\k}\arrow{w,t,A}{p}
   \node{A}\arrow{s,r,V}{\k}\arrow{w,t,V}{i}\\
 \node{(B_0/A)^{\astwo}}\node{B_0^{\astwo}}\arrow{w,t,A}{p^{\astwo}}
  \node{A^{\astwo}}\arrow{w,t,V}{i^{\astwo}}
\end{diagram}
\EE
where the vertical arrows correspond to the canonical
$B_0$-linear embeddings into the second dual.
Suppose $s\Cl B_0^{\astwo}\to A^{\astwo}$ is a bounded
$B_0$-linear map such that $s\c i^{\astwo}=\id_{A^{\astwo}}$.
Let $\ep=(s\c\k)(1_{B_0})\in A^{\astwo}$.
Since $s\c\k\Cl B_0\to A^{\astwo}$ is $B_0$-linear,
it is necessarily of the form
\[
b\mapsto(s\c\k)(b)=\ep b\qquad(b\in B_0)
\]
and
\[
\ep a=(s\c\k\c i)(a)=(s\c i^{\astwo}\c\k)(a)=\k(a)
 \qquad(a\in A).
\]
Every Banach space $X$ is dense in its second
dual in the weak$^{\ast}$ topology.
Indeed, for any $\th\in X^{\astwo}$,
and any linearly independent $n$-tuple
$\xi_1,\dots,\xi_n\in X^{\ast}$,
there exists $x\in X$ such that
\[
\th(\xi_1)=\xi_1(x),\qquad\dots\qquad,\ \th(\xi_n)=\xi_n(x).
\]
This follows from the fact that the linear mapping
\[
\BM{c}\xi_1\\\vdots\\\xi_n\EM\Cl X\LA F^n
\]
is surjective ($F$ denotes the ground field, either $\R$ or $\C$).
Any net $(e_i)_{i\in I}$ in $A$ which converges to $\ep$
in the weak$^{\ast}$ topology is a left approximate identity
in $A$. Since the unit ball of $A$ is weak$^{\ast}$-dense
in the unit ball of $A^{\astwo}$ (Goldstine's theorem,
cf.,\ e.g.\ \cite[V.4.5]{DS}) one can find $(e_i)_{i\in I}$
with $\sup_{i\in I}\|e_i\|=\|\ep\|$. This completes the
proof that (b) implies (c).

\medskip
Given any bounded left approximate identity on $A$
and a unital Banach algebra $B$ which contains $A$
as a closed right ideal, let $(l_i)_{i\in I}$ be the
corresponding net of bounded $B$-linear maps
\[
l_i\Cl B\LA A,\qquad b\LMT l_i(b)=e_ib\qquad(b\in B).
\]
The net of adjoint maps $l_i^{\ast}\Cl A^{\ast}\to B^{\ast}$ is
bounded, hence, by the Banach-Alaoğlu theorem,
there exists a subnet $(l_i^{\ast})_{i\in I'}$ which
converges in the weak$^{\ast}$ topology of
$\L(A^{\ast},B^{\ast})\simeq(A^{\ast}\pT B)^{\ast}$ to a map
$\l\in\L(A^{\ast},B^{\ast})$. That map, being the limit
of $B$-linear maps is $B$-linear, and
\[
(\l(\a))(a)=\lim_{i\in I'}\,(l_i(\a))(a)
 =\lim_{i\in I'}\,\a(e_ia)=\a(a)
 \qquad(\a\in A^{\ast},\,a\in A),
\]
shows that $i^{\ast}\c l=\id_{A^{\ast}}$.
\end{proof}

\begin{cor}\label{cor:BAI}
Let $A$ be a Banach algebra with bounded left
approximate identity. Then, for any Banach
algebra $B$ which contains $A$ as a closed
right ideal,
\BE{B A}
\begin{diagram}
 \node{B/A}\node{B}\arrow{w,,A}\node{A}\arrow{w,,V}
\end{diagram}
\EE
is a pure-exact extension of Banach spaces.
\end{cor}

The extension in \eqref{B A} is pure-exact if it is pure
exact when $B$ is replaced by its unitalization
and for unital $B$ the assertion in \ref{cor:BAI}
is an immediate corollary of Proposition
\ref{pr:BAI}.

\medskip
Note that Proposition \ref{pr:BAI} is for Banach algebras
what Proposition 2 in \cite{Wodzicki.PNAS} is for rings.

\end{document}